\documentclass[12pt]{article}
\usepackage{setspace}
\usepackage{geometry}
\usepackage{amsfonts, amssymb, amsmath, amsxtra, latexsym, amsthm, amscd}
\usepackage[mathscr]{eucal}

\usepackage{graphicx}
\usepackage{epstopdf}
\usepackage{pgfplots}
\pgfplotsset{compat=1.18}
\usepackage{tikz}
\usepackage{float}
\usepackage{booktabs}
\usepackage{tabularx}
\usepackage{subcaption}
\usepackage{caption}

\usepackage[dvipsnames]{xcolor} 
\usepackage{soul}
\usepackage{enumitem}
\usepackage{url}
\usepackage{fancyhdr}
\usepackage{makeidx}

\RequirePackage[colorlinks,citecolor=blue,filecolor=black,linkcolor=blue,urlcolor=black]{hyperref}

\usepackage[round]{natbib}

\newtheorem{theorem}{Theorem}

\newtheorem{case}{Special case}

\newtheorem{corollary}{Corollary}

\newtheorem{definition}{Definition}

\newtheorem{lemma}{Lemma}

\newtheorem{proposition}{Proposition}
\newtheorem{remark}{Remark}

\newtheorem{assum}{Assumption}

\newcommand{\rr}{\mathbb{R}}

\newcommand{\NSP}[1]{{\color{blue} NSP: #1}}

\begin{document}
\title{From Balanced Growth to Innovation Cycles:
A Solow-Type Semi-Endogenous Growth Model\footnote{We would like to thank David Desmarchelier, an Associate Editor and two anonymous referees for their helpful comments and suggestions.}}
\author{Cuong LE VAN\thanks{PSE, CNRS, IPAG Business School, FTU Vietnam. Email: Cuong.Le-Van@univ-paris1.fr.}\and Tin NGUYEN-TRONG\footnote{EconomiX, University of Paris Nanterre.  Email: tin.nt2@vnp.edu.vn.}  \and Ngoc-Sang PHAM\thanks{BETA, University of Lorraine; EM Normandie Business School, M\'etis Lab; EconomiX, University of Paris Nanterre;  Emails: pns.pham@gmail.com, ngoc-sang.pham@univ-lorraine.fr.} \and Thi Kim Cuong PHAM\footnote{EconomiX, University of Paris Nanterre. Email: pham$\_$tkc@parisnanterre.fr.}}
\date{\today}

\maketitle
\begin{abstract}

We develop a growth model with an R\&D sector in which the elasticity of innovation with respect to existing knowledge can be negative. We investigate the existence and uniqueness of a balanced growth path (BGP) and derive closed-form growth factors, showing that productivity and output growth are semi-endogenous and population-driven. We also establish the global stability under a general production function. Under testable parameter restrictions, the economy converges to the BGP. However, when the knowledge elasticity is sufficiently low, a Jacobian-based condition implies instability, and an N-period innovation cycle can emerge. Comparative statics are conducted to explore the impact of several factors, including research efficiency, input elasticity, and population growth. Numerical simulations map out transitions and mark the points at which stability is lost. They make clear when innovation frictions endanger balanced growth and provide practical guidance for R\&D and demographic policy.

\vspace{1em}


\end{abstract}


\section{Introduction}\label{section1}
Technological progress is central to sustained economic growth. In the neoclassical model of \citet{solow1956}, however, technological change is exogenous and therefore remains outside the mechanism of the model. Subsequent research has explained technological progress through learning, human-capital accumulation, and purposeful research. \citet{Arrow1962} emphasizes learning-by-doing, \citet{romer1986} studies the increasing returns generated by nonrival knowledge, and \citet{Romer1990} models the production of new ideas through deliberate R\&D activity. These contributions establish that knowledge can sustain growth because it raises productivity and facilitates further innovation.

More recent evidence nevertheless suggests that the relationship between existing knowledge and subsequent innovation need not be positive. \citet{Bloom2020} analyze substantial declines in research productivity across a range of industries and technologies: maintaining a given rate of technological progress has required rapidly increasing research effort. Their findings are consistent with the view that past discoveries can make the remaining problems more difficult to solve. This possibility raises a question that has received relatively less attention. If accumulated knowledge eventually makes further innovation sufficiently difficult, does the economy still converge smoothly to a balanced growth path, or can the innovation process itself generate instability and recurrent fluctuations?

We address this question in a discrete-time Solow-type model with endogenous innovation. A constant fraction of output is saved, and fixed shares of saving and labor are allocated between final-good production and research. The economy has two accumulated state variables: physical capital and a nonrival stock of disembodied technological knowledge. Final output is produced from knowledge, physical capital, and production labor. New knowledge is produced from inherited knowledge, a flow of R\&D expenditure, and research labor. The same knowledge stock can be used simultaneously in final production and research, whereas labor and saving must be allocated between the two sectors. The variable representing R\&D expenditure is a flow rather than a stock of human capital.

This structure makes the paper's position precise. Knowledge creation is endogenous, since research resources and inherited knowledge determine the production of new ideas. However, saving and sectoral allocations are exogenous rather than determined through intertemporal optimization. The model is therefore not a standard endogenous-growth model in which policy or preference parameters permanently determine the long-run growth rate. It is instead a Solow-type model with endogenous innovation and semi-endogenous long-run growth.

This distinction connects our framework to the non-scale R\&D literature. Early R\&D-based models, such as \citet{Romer1990}, imply strong scale effects because a larger research sector permanently raises the growth rate. \citet{Jones1995}, \citet{Segerstrom1998}, and \citet{Young1998} develop alternative mechanisms that remove these scale effects and make long-run growth depend on exogenous demographic forces and the parameters of the idea-production technology. \citet{Jonesreview} provides a recent synthesis of this semi-endogenous approach. Our model follows the same population-driven long-run mechanism: research efficiency, saving, and the allocation of resources to R\&D affect long-run level of output and transitional dynamics, but sustained growth ultimately depends on population growth and the elasticities of idea production.

Our balanced-growth analysis is also closely related to \citet{EicherTurnovsky1999}, who develop a general two-sector non-scale model with capital accumulation and endogenous technological change. They show that balanced growth can arise under different configurations of sectoral returns to scale and that final output and knowledge do not necessarily grow at the same rate. Their analysis provides a useful benchmark for our balanced-growth results. We impose a more specific production structure, which allows us to derive explicit growth factors and steady-state levels. We then go beyond the balanced-growth characterization and investigate questions that are not resolved by a balanced-growth characterization alone: global convergence from arbitrary positive initial conditions, loss of local stability, and deterministic innovation cycles. Our contribution lies thus in the dynamics around and away from the balanced growth path, rather than simply establishing the existence of population-driven balanced growth.

We first provide a complete characterization of the nondegenerate balanced growth path. Under economically meaningful feasibility conditions, the balanced growth path exists and is unique. Knowledge, capital, and output grow in the long run at rates determined by population growth. Parameters that increase research intensity or research efficiency affect long-run levels and may produce substantial transitional effects, but not the long-run growth rates. In this sense, the model separates the determinants of growth from the determinants of levels and transitional dynamics, as in the semi-endogenous literature.

We next study whether the economy converges to this balanced growth path. Once the growing variables are appropriately normalized, the model becomes a two-dimensional nonlinear dynamical system in capital and knowledge. When inherited knowledge has a nonnegative overall effect on the production of new ideas and the accumulation technologies have sufficiently strong diminishing returns, the normalized system is monotone and globally convergent. Thus, every trajectory starting from positive capital and knowledge converges to the balanced-growth steady state. We subsequently extend the result beyond the Cobb--Douglas specification to a broader class of production and innovation technologies satisfying appropriate balanced-growth, monotonicity, and concavity conditions. This result is stronger than local stability because it does not depend on the economy starting close to its balanced growth path.

A key contribution of the paper concerns the case in which inherited knowledge exerts a negative overall effect on the production of new ideas. It is important here to distinguish the direct and indirect channels. For a given level of R\&D expenditure, a larger knowledge stock may make new discoveries harder to obtain. At the same time, higher knowledge raises final output and hence the resources available for R\&D. The dynamics depend on the combined effect of these two channels. A negative direct effect alone does not necessarily destabilize the economy if the resulting increase in R\&D expenditure is strong enough. When this negative feedback is moderate, the balanced-growth steady state remains locally stable.  If it becomes sufficiently strong, the steady state can lose stability through oscillations. When knowledge temporarily exceeds its balanced-growth level, for example, the difficulty of producing new ideas may become so severe that new knowledge no longer compensates for obsolescence. The knowledge stock then falls below its balanced-growth level, making subsequent discoveries easier and causing an overshoot in the opposite direction. The interaction between the difficulty of producing new ideas and knowledge obsolescence can therefore transform smooth convergence into alternating periods of high and low innovation.

In a tractable special case, we establish analytically the existence of two-period innovation cycles and, under a stronger sufficient condition, periodic innovation paths of every prescribed period. These results suggest that the difficulty of finding new ideas can be a source not only of slower growth but also of endogenous fluctuations. They are existence results: periodic orbits need not attract all nearby trajectories, and our numerical analysis distinguishes attracting cycles from repelling periodic paths. We also emphasize that the analytical cycle result depends on a special configuration of the idea-production technology.  The parameter values used to illustrate complex dynamics constitute a theoretical stress test rather than an empirical calibration.

Our mechanism is related to, but distinct from, the literature on endogenous growth fluctuations. \citet{Guesnerie_Woodford_1993} survey mechanisms through which nonlinearities, expectations, and market imperfections can generate endogenous fluctuations. \citet{walde05} obtains growth cycles when uncertain R\&D success interacts with investment decisions. Most closely related to our discussion, \citet{EvansHonkapohjaRomer1998} construct  rational-expectations growth cycles in which complementarity between capital goods supports high- and low-growth regimes and an extrinsic random variable coordinates self-fulfilling switches between them. Their cycles depend on expectational indeterminacy and stochastic regime changes. Our cycles are instead deterministic and backward-looking. They arise from the nonlinear law of motion of knowledge around a unique balanced growth path and require neither extrinsic uncertainty nor self-fulfilling changes in expectations.

The model also has formal connections with the literature on human-capital accumulation, but the economic interpretation is different. \citet{MankiwRomerWeil1992} extend the Solow model by allowing a fixed share of output to accumulate a rival stock of human capital. A special parameterization of our accumulation equation has a similar mathematical form. In our model, however, the accumulated variable is nonrival technological knowledge rather than rival human capital. Likewise, \citet{Uzawa1965} and \citet{lucas1988} study two-sector models in which agents allocate time or human capital between current production and education. Their allocation decisions are endogenous, and human capital is a rival productive input. By contrast, our sectoral shares are fixed, the knowledge stock is nonrival, and the R\&D variable is a flow of research expenditure rather than an accumulated stock. These distinctions are important both for the interpretation of returns to scale and for the classification of our long-run growth mechanism.


The remainder of the paper is organized as follows. Section~\ref{section2} presents the model and clarifies the economic interpretation of knowledge and R\&D expenditure. Section~\ref{section3} characterizes the balanced growth path and develops the comparative statics. Section~\ref{section4} studies global convergence, local stability, instability, and innovation cycles. Section~\ref{section5} presents the numerical illustrations. Section~\ref{conclusion} concludes. Formal proofs are collected in the Appendix.

\section{A growth model with endogenous innovation}
\label{section2}
We develop a discrete-time dynamic growth model in which technological progress is endogenously determined by investment in research and development (R\&D). 
Assume that there is one consumption good and there are two sectors: a final goods sector and an R\&D sector. 

\paragraph{Consumption, production, and savings.}
The first sector produces the consumption good by using physical capital $K_t$, labor $L_{y,t}$, and knowledge $A_t$. 
Assume that the aggregate output is determined by
\begin{align}\label{production_function}
Y_t = F(A_t,K_t,L_{y,t}).
\end{align}
Initial values $K_0, A_0>0$ are exogenous.  We consider the Cobb--Douglas production function
\begin{assum}\label{functionF}$F(A,K,L_{y})=A^{\alpha_a} K^{\alpha_k} L_{y}^{\alpha_\ell}$, where \( \alpha_a, \alpha_k, \alpha_\ell \geq 0 \).
\end{assum}

\noindent A fixed fraction \( s \in [0,1] \) of the output is saved and invested each period so that the total saving is equal to $S_t = sY_t.$ 
\( S_t \) will later be allocated between capital investment (with a fraction $1 - \gamma$) and R\&D expenditure (with a fraction $\gamma$).

\paragraph{Population and labor allocation.} The total labor force \( L_t \) grows at a constant gross rate \( G_n > 0 \), starting from the initial level \( L_0 >0\): $L_t = L_0 G_n^t$ $\forall t$.

The labor force is allocated into two sectors: a sector producing the final good and a so-called R\&D sector: $L_{y,t} = \lambda_y L_t,\quad L_{a,t} = \lambda_a L_t,\quad \lambda_a + \lambda_y = 1,  \lambda_a\in [0,1]$, where \( L_{y,t} \) is used in the final good production, and \( L_{a,t} \) consists of people (researchers, scientists, engineers, ...) working in R\&D. The shares \( \lambda_y\geq 0 \) and \( \lambda_a \geq 0\) are constant and sum to one.

\paragraph{Capital accumulation.} The evolution of the physical capital stock is determined by a standard accumulation equation:
$K_{t+1} = (1 - \delta_k) K_t + (1 - \gamma) S_t.$ 
 Note that \( (1 - \gamma) \) corresponds to a fraction of the total savings that is allocated to capital investment, and capital depreciates at a rate \( \delta_k \in [0,1] \).

\paragraph{R\&D expenditure.} The remaining fraction \( \gamma \in [0,1] \) of savings is allocated to R\&D expenditure:
$H_{t+1}= \gamma S_t.$  The variable $H_{t+1}$ represents the flow of R\&D expenditure undertaken in period $t$, which contributes to the production of new knowledge available in the next period $t+1$, as shown in equation \eqref{A_accumulation} below. It can be interpreted as current research resources devoted to R\&D and innovation, including laboratory experiments, software development, data acquisition, patenting activities, clinical trials, ...

\paragraph{Knowledge stock.} We assume that the effective stock of ideas/knowledge is determined by
\begin{align}\label{A_accumulation}
A_{t+1} &= (1 - \delta_a) A_t + F_a(A_t, H_{t+1}, L_{a,t+1}). 
\end{align}
 The accumulation of the knowledge stock \( A_{t+1} \) is based on three components: (1)  retained knowledge from the previous period $A_t$, depreciated at a rate \( \delta_a \in [0,1] \), (2) new knowledge generated through a knowledge production function \( F_a \), which depends on the existing knowledge \( A_t \),  R\&D expenditure \( H_{t+1} \), and (3) the number of scientists \( L_{a,t+1} \). 

The knowledge stock $A_t$ is interpreted as the aggregate stock of disembodied technological knowledge, such as scientific discoveries, production know-how, engineering designs, organizational knowledge, etc. Knowledge is a nonrival input: the same stock can be used simultaneously in the production of final goods (see \eqref{production_function}) and in the creation of new knowledge (see \eqref{A_accumulation}).\footnote{However, our model is aggregate and does not explicitly model the ownership or pricing of ideas.}

Unless stated otherwise, we work under the following setup.
\begin{assum}\label{functionFa}
$F_a(A, H, L_{a}) = \mu_aA^{\gamma_a}H^{\gamma_h}L_{a}^{\gamma_\ell}$, where $\gamma_h,\gamma_\ell\geq 0$.
\end{assum}

The function $F_a$ takes a Cobb--Douglas form, where $\mu_a>0$ is a productivity parameter, and the exponents $\gamma_h,\gamma_\ell\ge0$ measure the elasticities of innovation with respect to R\&D expenditure and research labor, respectively. 

The parameter $\gamma_a$ captures the direct effect of the inherited knowledge stock on the production of new ideas, holding other research inputs fixed (for an indirect effect, see Remark \ref{indirect_effect} below). It may be either positive or negative. A positive value ($\gamma_a>0$) reflects knowledge spillovers, whereby a larger stock of existing knowledge facilitates further innovation. By contrast, a negative value ($\gamma_a<0$) captures the idea that, conditional on a given level of R\&D expenditure and research labor, the existing knowledge stock makes the discovery of additional ideas progressively more difficult, consistent with the notion that ``ideas are getting harder to find'' \citep{Bloom2020}.

To sum up, we have the following economic system:
\begin{subequations}\label{system}
\begin{align}
    K_0, A_0,L_0& \text{ are given and positive }\\
\label{2b}   \text{Output production: }  Y_t&=A_t^{\alpha_a}K_t^{\alpha_k}L_{y,t}^{\alpha_\ell}\\
    \text{Saving: } S_t&=sY_t\\
    \text{Population: }L_{t}&=L_0G_n^t\\
    \text{ Sectoral labor: } L_{y,t}&=\lambda_yL_t, \quad  L_{a,t}=\lambda_aL_t, \quad \lambda_a+\lambda_y=1\\
     \text{Physical capital: }   K_{t+1}&= (1-\delta_k)K_t+(1-\gamma)S_t\\
  \label{2g}  \text{R\&D expenditure: }  H_{t+1}&=\gamma S_t\\
\label{dynamicsA}\text{Knowledge: }  A_{t+1}&=(1-\delta_a)A_t+\mu_aA_t^{\gamma_a}H_{t+1}^{\gamma_h}L_{a,t+1}^{\gamma_\ell}
\end{align}
\end{subequations}

\begin{remark}\label{indirect_effect}
Substituting the R\&D expenditure $H_{t+1}=\gamma sY_t
=\gamma sA_t^{\alpha_a}K_t^{\alpha_k}L_{y,t}^{\alpha_\ell},$
into the knowledge accumulation equation \eqref{dynamicsA} yields
\begin{align}\label{AGamma}
A_{t+1}
&=(1-\delta_a)A_t
+\big[\mu_a(\gamma s)^{\gamma_h}
\lambda_y^{\alpha_\ell\gamma_h}
\lambda_a^{\gamma_\ell}
G_n^{\gamma_\ell}\big]
A_t^{\gamma_a+\alpha_a\gamma_h}
K_t^{\alpha_k\gamma_h}
L_t^{\alpha_\ell\gamma_h+\gamma_\ell}.
\end{align}
Therefore, the effective elasticity of new knowledge with respect to the inherited knowledge stock is $\Gamma\equiv\gamma_a+\alpha_a\gamma_h.$ Hence, while $\gamma_a$ measures the \emph{direct} effect of the inherited knowledge stock on new knowledge, holding R\&D expenditure fixed, the effective elasticity $\Gamma$ also incorporates the indirect channel through output and endogenous R\&D expenditure.
\end{remark}
\begin{remark}[Returns to scale]\label{remark:returns}
Since $A_t$ is nonrival in our baseline interpretation, the natural constant-returns benchmark for final-good production concerns the rival inputs: $\alpha_k+\alpha_\ell=1$. By contrast, $\alpha_a+\alpha_k+\alpha_\ell=1$ means that $F$ is homogeneous of degree one when knowledge, capital, and labor are all scaled together. We impose neither restriction. The case $\alpha_a=\alpha_\ell$ corresponds to the labor-augmenting form $F(A,K,L_y)=K^{\alpha_k}(AL_y)^{\alpha_\ell}.$

Similarly, for the R\&D technology, $\gamma_h+\gamma_\ell=1$ corresponds to constant returns in the contemporaneous rival research inputs $H$ and $L_a$. The parameter $\gamma_a$ instead measures the effect of inherited knowledge on new-idea production and may be positive or negative.
\end{remark}


\paragraph{Some special cases.} Before investigating the dynamics of the economy, it is useful to mention some special cases of our framework.
\begin{case}{\normalfont
When $\gamma_h=\gamma_\ell=0, \gamma_a=1$ and $\mu_a > 0$, equation (\ref{dynamicsA}) becomes $A_{t+1}=(1-\delta_a+\mu_a)A_t$. We recover the Solow model with exogenous growth of technology ``falling from the sky" and its net growth rate is given by $\mu_a - \delta_a $. The variable $A_t$ represents productivity.}
\end{case}

\begin{case}{\normalfont
When $\delta_a=0=\gamma_h$, equation (\ref{dynamicsA}) becomes $\frac{A_{t+1}-A_t}{A_t}=(\mu_a\lambda_a^{\gamma_\ell})A_t^{\gamma_a-1}L_{t+1}^{\gamma_\ell}$, which can be viewed as a discrete-time version of the model of ideas in \citet[Section B]{JonesAER2022}. Following \citet{JonesAER2022}, variable $A_t$ in our framework can be interpreted as the idea. The value $\gamma_a-1$ in our paper corresponds to $-\beta<0$ in equation (5) in \citet{JonesAER2022}.
}
\end{case}
\begin{case}\label{special3}
\normalfont {When $\mu_a=1, \gamma_a=\gamma_{\ell}=0$ and $\gamma_h=1$, equation \eqref{dynamicsA} becomes $A_{t+1}=(1-\delta_a)A_t+\gamma sY_t$. So, the variable $A_t$ can be interpreted as the human capital stock as in \citet{MankiwRomerWeil1992}. By the way, we recover a discrete-time version of the augmented Solow model in \citet[Section 2.A]{MankiwRomerWeil1992}, which includes accumulation of both physical and human capitals.\footnote{The growth factor $G_n$ in our model corresponds to n+g in \citet[Section 2.A]{MankiwRomerWeil1992}. Rewriting their model using our notation, their equation (8) requires that $\alpha_{\ell}=1-\alpha_k-\alpha_a$.}}
\end{case}

For the remainder of the paper, without an explicit statement, we require the following assumption.
\begin{assum}\label{basicassumptions}$G_n>0$, $s\in (0,1]$, $\delta_k,\delta_a\in (0,1]$, $\gamma\in [0,1)$,\footnote{When $\gamma=0$, we adopt the convention that $\gamma_h=0$ and remove $H_t$ from the model.} $\lambda_a\in (0,1)$,  $\alpha_k<1$, $\lambda_y\in (0,1]$, $\mu_a>0$.
\end{assum}

\section{Balanced Growth Path and Steady-State}\label{section3}

We study the dynamics of our economic variables. We first introduce the notion of a Balanced Growth Path.

\subsection{Balanced Growth Path: existence and characterization}

\begin{definition}[\emph{Balanced Growth Path and Steady State}]
Given $L_0>0,G_n > 0$, a balanced growth path (BGP) is a sequence 
$(K_t, L_t, L_{y,t}, L_{a,t}, H_t, A_t, Y_t)$ 
that satisfies the system (\ref{system}) and $X_t = X_* G_x^t>0 \text{ } \forall X \in \{K,A,Y\}, H_{t+1}=H_*G_h^t\geq 0$ $\forall t$, where $G_x > 0$ corresponds to the gross rate of growth of the variable $X$ and $g_x=G_x-1 $ the net rate of growth, which may be positive or negative. 

The constant $X_*$ is called the normalized steady-state value (or simply the steady-state value) of the variable $X$, for $X\in\{K,H,A,Y\}$.
\end{definition}



Before characterizing the BGP, we introduce some notations and show a necessary condition for the existence of a BGP. Denote \begin{subequations}
\begin{align}\phi_a&\equiv(1-\gamma_a)(1-\alpha_k)-\alpha_a\gamma_h,&\phi_n&\equiv\alpha_\ell\gamma_h+(1-\alpha_k)\gamma_\ell,\\
 \label{xkxa}
X_a &\equiv \mu_a (\gamma s)^{\gamma_h} (\lambda_y)^{\alpha_\ell \gamma_h} (\lambda_a G_n)^{\gamma_\ell} L_0^{\alpha_\ell \gamma_h + \gamma_\ell}, & X_k &\equiv (1 - \gamma) s \lambda_y^{\alpha_\ell} L_0^{\alpha_\ell}.
 \end{align}
 \end{subequations}

\begin{lemma}\label{lemma1}
Consider the system (\ref{system}). Let Assumptions \ref{functionF}-\ref{basicassumptions} be satisfied. If there exists a BGP, then we have $G_k > 1-\delta_k$, $G_a > 1 - \delta_a$,
\begin{subequations}
\begin{align}
\label{Gky}G_k &= G_y,  \\
\label{Gank}G_k^{1-\alpha_k}&=G_a^{\alpha_a} G_n^{\alpha_\ell},\\
\label{gan}
 G_a^{\phi_a}&= G_n^{\phi_n}.
\end{align}
\end{subequations}
and the steady state values \( A_*, K_*, Y_* \) are given by 
\begin{subequations}
\begin{align}\label{formulaA_ss}
A_*^{\frac{(1-\gamma_a)(1-\alpha_k)-\alpha_a \gamma_h}{1 - \alpha_k}} 
&= \frac{X_a}{G_a - 1 + \delta_a}  \left( \frac{X_k}{G_k - 1 + \delta_k} \right)^{\frac{\alpha_k \gamma_h}{1 - \alpha_k}}, \\
\label{formulaK_ss}K_* &= \left( \frac{X_k A_*^{\alpha_a}}{G_k - 1 + \delta_k} \right)^{\frac{1}{1 - \alpha_k}},\\
 \label{formulaY_ss}Y_* &= \frac{(G_k - 1 + \delta_k) K_*}{(1 - \gamma) s}.
\end{align}
\end{subequations}
\end{lemma}
\begin{proof} See Appendix \ref{appendix1}.\end{proof}

We state the main result of this section, which completely characterizes the BGP.
\begin{theorem}[\emph{Characterization of BGP}]\label{theorem1}Let Assumptions \ref{functionF}-\ref{basicassumptions} be satisfied.

\begin{enumerate}
   \item \label{theorem1_1} Suppose 
$\phi_a\not=0$. Then, there exists a BGP if and only if
\begin{align}
\label{Ga*}
G_a^* &\equiv G_n^{\psi_a}
> 1 - \delta_a,   G_k^* \equiv G_n^{\psi_k} > 1 - \delta_k,
\text{ where }
\psi_a \equiv
\frac{\phi_n}{\phi_a},
\psi_k\equiv \frac{\alpha_\ell+\alpha_a\psi_a}{1 - \alpha_k}.  
\end{align}
%
Moreover, under \eqref{Ga*}, the BGP is unique. The growth rates \( (G_y, G_k, G_a) \) of the BGP are determined by
$G_a = G_a^*,  G_k = G_y = G_k^*$ 
while the steady state values \( A_*, K_*, Y_* \) are given by \eqref{formulaA_ss}, \eqref{formulaK_ss} and \eqref{formulaY_ss}.

\item Suppose $\phi_a\equiv(1-\gamma_a)(1-\alpha_k)-\alpha_a\gamma_h=0$.
\begin{enumerate}
 \item A necessary condition for a BGP is $G_n^{\phi_n}=1,$ 
 equivalently $G_n=1$ or $\phi_n=0$. If this condition fails, no BGP exists.
 \item Suppose $G_n^{\phi_n}=1$. Then, there is a continuum of BGPs indexed by $A_*\in (0,\infty)$ and each BGP is determined by \eqref{formulaK_ss},  \eqref{formulaY_ss} and 
 \begin{subequations}
\begin{align}\label{boundara_Ga}
  \big(G_a-(1-\delta_a)\big)\big[G_n^{\frac{\alpha_\ell}{1-\alpha_k}}
        G_a^{\frac{\alpha_a}{1-\alpha_k}}-(1-\delta_k)\big]^{\frac{\alpha_k\gamma_h}{1-\alpha_k}}=X_aX_k^{\frac{\alpha_k\gamma_h}{1-\alpha_k}},\\
\label{boundara_Ga2}      G_k= G_n^{\frac{\alpha_\ell}{1-\alpha_k}}
        G_a^{\frac{\alpha_a}{1-\alpha_k}}, \quad G_y=G_k.
  \end{align} 
   \end{subequations}
Note that this system \eqref{boundara_Ga}-\eqref{boundara_Ga2} has a unique solution  $(G_a,G_k,G_y)$ and it satisfies $G_a>1-\delta_a,G_k>1-\delta_k$. 
\end{enumerate}
\end{enumerate}
    
\end{theorem}
\begin{proof}
\begin{enumerate}
    \item When $\phi_a\not=0$, the necessity follows from Lemma~\ref{lemma1}. In particular, $G_a^{\phi_a}=G_n^{\phi_n}$ and $\phi_a\neq0$ give $G_a=G_n^{\phi_n/\phi_a}=G_a^*$. Substitution in \eqref{Gank} then gives $G_k=G_y=G_k^*$. The two strict inequalities in \eqref{Ga*} follow from the strictly positive flows of new knowledge and gross investment. 

Conversely, suppose \eqref{Ga*} holds. Equation \eqref{formulaA_ss} has a unique positive solution because its right-hand side is positive and its exponent is $\phi_a/(1-\alpha_k)\neq0$. Equations \eqref{formulaK_ss} and \eqref{formulaY_ss} then determine unique positive values of $K_*$ and $Y_*$, and $H_*=\gamma sY_*$. By the definition of $G_a^*$ and $G_k^*$ in \eqref{Ga*}, we get imply \eqref{Gank} and \eqref{gan}. Then,  substituting the resulting constants and growth factors into \eqref{system} verifies every equation for all $t$. This also proves uniqueness.

\item When $\phi_a=0$, condition \eqref{gan} implies that $G_n^{\phi_n}=1$.
Suppose now $G_n^{\phi_n}=1$. Since $\phi_a=0$, condition \eqref{formulaA_ss} implies that $ \big(G_a-(1-\delta_a)\big)\big[G_k-(1-\delta_k)\big]^{\frac{\alpha_k\gamma_h}{1-\alpha_k}}=X_aX_k^{\frac{\alpha_k\gamma_h}{1-\alpha_k}}$. By \eqref{Gank}, we have $G_k= G_n^{\frac{\alpha_\ell}{1-\alpha_k}}
        G_a^{\frac{\alpha_a}{1-\alpha_k}}$. Therefore, we get \eqref{boundara_Ga}, which has a unique solution $G_a$. 
\end{enumerate}
\end{proof}

Theorem~\ref{theorem1} characterizes the existence, uniqueness and multiplicity of balanced growth paths. In the generic case $\phi_a\neq0$, the conditions $G_n^{\psi_a}>1-\delta_a$ and $G_n^{\psi_k}>1-\delta_k$ ensure that the implied flows of new knowledge and gross physical investment are strictly positive. Equality in either condition would imply a zero flow, whereas a strict reversal would require a negative flow, both incompatible with strictly positive innovation and investment. These conditions do not require gross investment to cover depreciation: a BGP may feature declining stocks, provided that each stock declines more slowly than it would under depreciation alone.

The generic condition $\phi_a\equiv(1-\gamma_a)(1-\alpha_k)-\alpha_a\gamma_h\neq0$ rules out a singularity in the balanced-growth equations and ensures uniqueness whenever the two feasibility conditions hold. It is satisfied, with $\phi_a>0$, for example, when $\gamma_a+\gamma_h<1$ and $\alpha_k+\alpha_a<1$. In the degree-one homogeneous benchmark, where $\alpha_a+\alpha_k+\alpha_\ell=1=\gamma_a+\gamma_h+\gamma_\ell$, we have $\phi_a=\phi_n=\alpha_\ell\gamma_h+(1-\alpha_k)\gamma_\ell$. Provided that $\phi_n>0$, it follows that $\psi_a=\psi_k=1$. Thus, whenever the BGP exists, knowledge, capital, and output grow at the same rate as population: $G_a=G_k=G_y=G_n$.

On the knife-edge $\phi_a=0$, existence requires the compatibility condition $G_n^{\phi_n}=1$, equivalently $G_n=1$ or $\phi_n=0$. When this condition and the strict feasibility conditions are satisfied, the growth factors are uniquely determined, but the normalized steady-state levels form a continuum indexed by $A_*>0$, with $K_*$ determined by \eqref{formulaK_ss}. This multiplicity concerns BGPs associated with different admissible initial stocks, not multiple trajectories starting from the same initial conditions. For prescribed stocks $(A_0,K_0)$, the economy lies on a BGP from date zero only if $(A_0,K_0)=(A_*,K_*)$ for one member of this continuum.

\subsection{Comparative statics}

Theorem \ref{theorem1} also allows us to  provide comparative statics in order to understand the interplay between population, technological, and economic growth.
We now investigate how the steady states $Y_*, K_*, A_*$ and the growth rates depend on the parameters of the model.

\begin{proposition}[Population-driven long-run growth (semi-endogenous)]\label{prop1}
Suppose that the conditions of Theorem~\ref{theorem1}'s part \ref{theorem1_1}
hold and that $\phi_n>0$. 
At the BGP, the following statements hold.
\begin{enumerate}
 \item\label{prop11} The population elasticities of the growth factors are
 \begin{equation}\label{population_elasticities}
  \frac{\partial\ln G_a}{\partial\ln G_n}
   =\psi_a,\qquad
  \frac{\partial\ln G_y}{\partial\ln G_n}
   =\psi_k.
 \end{equation}
Moreover, $G_a,G_y\to1$ as $G_n\to1$ along values for which the BGP remains feasible.


 \item\label{prop12} The elasticities of the normalized steady state with respect to the saving
 rate $s$ and research efficiency $\mu_a$ are 
\begin{subequations}
\begin{align}
\label{population_elasticities1} \frac{\partial\ln A_*}{\partial\ln s}&=\frac{\gamma_h}{\phi_a}, & \frac{\partial\ln K_*}{\partial\ln s}&=\frac{1-\gamma_a}{\phi_a}, &\frac{\partial\ln Y_*}{\partial\ln s}&=\frac{\alpha_k(1-\gamma_a)+\alpha_a\gamma_h}{\phi_a},\\
\label{population_elasticities2}  \frac{\partial\ln A_*}{\partial\ln \mu_a}&=\frac{1-\alpha_k}{\phi_a}, & \frac{\partial\ln K_*}{\partial\ln \mu_a}&=\frac{\alpha_a}{\phi_a}, &\frac{\partial\ln Y_*}{\partial\ln \mu_a}&=\frac{\alpha_a}{\phi_a}.
\end{align}
 \end{subequations}
\end{enumerate}
 \end{proposition} 
\begin{proof}\eqref{population_elasticities} is a direct consequence \eqref{Ga*}. According to \eqref{formulaA_ss}-\eqref{formulaY_ss}, we can compute \eqref{population_elasticities1} and \eqref{population_elasticities2}.
\end{proof}
In the generic case $\phi_a\neq0$, the growth factors depend on population growth and the elasticities of the production and innovation technologies, but not on the saving rate $s$, research efficiency $\mu_a$, or the fixed allocation shares $\gamma$, $\lambda_a$, and $\lambda_y$. These latter parameters affect the levels of the BGP rather than its growth rates. This separation is the central semi-endogenous growth implication of the model \citep{Jones1995,Segerstrom1998,Young1998}, in contrast to the strong scale effects in \citet{Romer1990}.

According to the point \ref{prop11} of Proposition \ref{prop1},  long-run growth is mainly driven by population growth and the direction of the demographic effect depends on the sign of $\phi_a$. When $\phi_a>0$, both $\psi_a$ and $\psi_k$ are nonnegative, so faster population growth weakly raises the growth factors of knowledge and output. The effect on knowledge growth is strictly positive whenever $\phi_n>0$. By contrast, when $\phi_a<0$ and $\phi_n>0$, knowledge growth is decreasing in population growth, whereas the response of output growth depends on the sign of $\psi_k=(\alpha_\ell+\alpha_a\psi_a)/(1-\alpha_k)$. Thus, outside the usual regime $\phi_a>0$, the demographic effects on knowledge and output need not have the same sign. If $\phi_n=0$, then $\psi_a=0$ and $G_a=1$, independently of population growth. In all cases with $\phi_a\neq0$, both $G_a$ and $G_y$ converge to one as $G_n\to1$ along feasible BGPs.

Point \ref{prop12} of Proposition \ref{prop1} indicates that the effects of saving rate $s$ and the research efficiency $\mu_a$ on the steady state levels also depend on $\phi_a$. When $\phi_a>0$, a higher saving rate or greater research efficiency weakly raises the normalized steady-state levels $A_*$, $K_*$, and $Y_*$. The effects are strict whenever the corresponding technological channel is active; in particular, research efficiency strictly raises $A_*$, while saving strictly raises $A_*$ when $\gamma_h>0$. Greater research efficiency generates more knowledge from given research inputs, while higher saving supplies additional resources for physical investment and R\&D. Through the interaction between knowledge, production, and accumulation, these changes shift the BGP upward without altering its growth factors.

When $\phi_a<0$, these level effects can be reversed. Greater research efficiency strictly lowers $A_*$ and weakly lowers $K_*$ and $Y_*$, while a higher saving rate weakly lowers $A_*$. The saving effects on capital and output are determined, respectively, by $(1-\gamma_a)/\phi_a$ and $[\alpha_k(1-\gamma_a)+\alpha_a\gamma_h]/\phi_a$, and therefore need not have the same sign. These results describe changes in the stock levels required to sustain the balanced-growth factors; they do not imply that an increase in saving or research efficiency causes an economy starting from given stocks to converge to a lower path. Such a conclusion requires a separate analysis of stability and transitional dynamics, undertaken in Section~\ref{section4}. Note that, by  Proposition \ref{unstable}, condition $\phi_a<0$ implies instability, whenever the positive steady state exists.

Notice also that the technological elasticities govern the strength of demographic transmission. Within the region $\phi_a>0$, the expressions for $\psi_a$ and $\psi_k$ imply that both exponents are weakly increasing in each of $\alpha_k,\alpha_a,\alpha_\ell,\gamma_a,\gamma_h$, and $\gamma_\ell$, holding the other parameters fixed. Consequently, higher elasticities weakly raise the BGP growth factors when $G_n>1$ and weakly lower them when $0<G_n<1$. This monotonicity need not extend to the region $\phi_a<0$. Importantly, technological elasticities should be distinguished from allocation shares: the former determine the growth exponents, whereas the latter affect BGP levels without changing these exponents.

Overall, the model combines the semi-endogenous insight that population growth $G_n$ is fundamental driver of long-run growth with an R\&D-based mechanism in which allocation and technological elasticities determine how strongly demographic growth is transmitted into productivity growth.  Within the conventional region \(\phi_a>0\) and \(\phi_n>0\), \(\psi_a>0\); hence, when \(G_n>1\), faster population growth raises long-run productivity growth. When \(\phi_a<0\), the sign can reverse, as discussed above

In this sense, $\psi_a$ and $\psi_k$ are amplification terms: they summarize the interaction between population growth, capital accumulation, and knowledge production, and provide a bridge between the semi-endogenous framework of \citet{Jones1995}, \citet{kremer1993} and R\&D-based growth models such as \citet{Aghion1992, Jones2001} and with recent evidence on diminishing returns to ideas \citep{Bloom2020}.

The following result investigates the relationship between population growth and economic growth.
\begin{corollary}[Demography, aggregate growth, and per-capita output]
\label{BGPcompare1}
Suppose that the conditions of Theorem~\ref{theorem1}'s part \ref{theorem1_1}
hold and that $\phi_n>0$. 

\begin{enumerate}
    \item Let $G_n > 1$. Then,  $G_a > 1$ iff\footnote{Iff means that "if and only if".} $(1 - \gamma_a)(1 - \alpha_k) > \alpha_a \gamma_h$ while $G_y>1$ iff $\psi_k>0$. Moreover, $G_a>G_n$ iff $\psi_a>1$ and $G_y>G_n$ iff $\psi_k>1$;
    \item Let $G_n < 1$.  Then, $G_a > 1$ iff $(1 - \gamma_a)(1 - \alpha_k) < \alpha_a \gamma_h$ while $G_y>1$ iff $\psi_k<0$. Moreover, $G_a>G_n$ iff $\psi_a<1$ and $G_y>G_n$ iff $\psi_k<1$.
    \end{enumerate}
\end{corollary}

\begin{proof}
By Theorem~\ref{theorem1}, $\ln G_a=\psi_a\ln G_n, \ln G_y=\psi_k\ln G_n,
 \psi_a={\phi_n}/{\phi_a},$ with $\phi_n>0$. Hence $\psi_a>0$ iff $\phi_a>0$, i.e.,
$(1-\gamma_a)(1-\alpha_k)>\alpha_a\gamma_h$. Moreover, $\ln({G_a}/{G_n})=(\psi_a-1)\ln G_n, 
\ln({G_y}/{G_n})=(\psi_k-1)\ln G_n.$ 
The claims follow immediately, noting that $\ln G_n>0$ if $G_n>1$ and
$\ln G_n<0$ if $0<G_n<1$.
\end{proof}
Corollary \ref{BGPcompare1}'s point 1 complements the classical view of Malthus (1798) (see \citet{Malthus1989}), in which population growth absorbs productivity gains and keeps living standards at subsistence. 
Our framework thus extends the Malthusian mechanism by showing that knowledge creation and factor allocation can amplify growth beyond the limits imposed by demographic expansion.

We now look at point 2 of Corollary \ref{BGPcompare1}, where $G_n<1$, meaning that the population decreases. In this case, $G_a>G_n$ iff $\psi_a<1$ while $G_y>G_n$ iff $\psi_k<1$. In this case, productivity and output 
can still outpace the declining population, provided that the 
R\&D channel is sufficiently strong. This corresponds to a situation in which technological progress compensates for demographic decline,
a possibility emphasized in recent discussions of low-population
growth economies \citep{Bloom2020, Jones2001, JonesAER2022}.

\section{Dynamic Analysis}\label{section4}
We now investigate the dynamic properties of the economic system and of the BGP. Unless explicitly stated otherwise, Assumptions \ref{functionF}--\ref{basicassumptions} are maintained throughout.
\subsection{Stability of the economic system}
\label{section-stability}
Let us start by introducing the notions of stability and instability.
\begin{definition}    Consider a function $f: \rr^n\to \rr^n$. 
\begin{enumerate}
    \item A fixed point $x^*$ is said to be stable if, for any $\epsilon>0$, there exists $\delta>0$ such that for any $x$ satisfying $\|x-x^*\|<\delta$, we have $\|f^t(x)-x^*\|<\epsilon$ for any $t=1,2,\ldots$, where the sequence $(f^t(x))_t$ is defined by $f^0(x)=x$ and $f^{t+1}(x)\equiv f(f^t(x))$.\footnote{Here, we consider the Euclidean  distance, i.e., $\|(a_1,\ldots,a_n)-(b_1,\ldots,b_n)\|\equiv \sqrt{(a_1-b_1)^2+\cdots+ (a_n-b_n)^2} $.}
    
\item  $x^*$ is said to be {\it attracting} if there exists $\eta>0$ such that: $\| x-x^*\|<\eta $ implies $\lim_{t\to \infty}f^t(x)=x^*$. It is called {\it global attracting} if $\eta=\infty$.

\item $x^*$ is said to be an {\it asymptotically stable point} if it is stable and attracting.  If $\eta=\infty$, then $x^*$ is said to be {\it globally asymptotically stable}.
\end{enumerate}
\end{definition}

We now come back to our economic system.  Denote the  detrended variables 
\begin{equation}\label{hat_variables}
    \hat{K}_t = \frac{K_t}{G_k^t}, \quad \hat{A}_t = \frac{A_t}{G_a^t}, \quad \hat{Y}_t = \frac{Y_t}{G_y^t},
\end{equation}
where $G_a, G_k, G_y$ are the growth rates of the BGP.

We aim to investigate the conditions under which $(\hat{K}_t, \hat{A}_t, \hat{Y}_t)$ converges to $(K_*, A_*, Y_*)$. Since $\hat{Y}_t$ can be easily computed as a function of $(\hat{K}_t, \hat{A}_t)$, it suffices to focus on $\hat{K}_t, \hat{A}_t$. 
According to the system (\ref{system}), combined with the change of variables (\ref{hat_variables}) and the balanced-growth conditions \eqref{Gky}-\eqref{gan} of Lemma \ref{lemma1}, we have a two-dimensional dynamical system of $(\hat{K}_t, \hat{A}_t)$.
\begin{lemma}[Detrended dynamics]\label{lemma-2dim-system}
The normalized stocks satisfy the following dynamical system.
    \begin{subequations}
\label{system-tilde}
 \begin{align}\label{system-tildeK}
  G_k\hat{K}_{t+1}&= (1-\delta_k)\hat{K}_t+X_k\hat{A}_t^{\alpha_a}\hat{K}_t^{\alpha_k}\\
 \label{system-tildeA} G_a\hat{A}_{t+1}&=(1-\delta_a)\hat{A}_t+X_a\hat{A}_t^{\gamma_a+\alpha_a\gamma_h}\hat{K}_t^{\alpha_k\gamma_h}.
\end{align}
\end{subequations}
\end{lemma}
\begin{proof} See Appendix \ref{section-stability-proofs}.\end{proof}

\subsubsection{Global stability of the BGP}

The system (\ref{system-tilde}) is non-linear and two-dimensional. In general, it is not easy to obtain the global stability of a non-linear multi-dimensional system. However, in our cases, we manage to prove the convergence of $(\hat{K}_t, \hat{A}_t)$. The following result is a key.

\begin{lemma}[global convergence of dynamical systems with concave functions]
\label{convergence_general}
Consider a system $x_{t+1}=f(x_t)$ where $f=(f_1,\ldots,f_n)$, i.e.,
\begin{align}
    x_{j,t+1}&=f_j(x_{1,t},\ldots,x_{n,t}) \text{ } \forall j\in \{1,2,\ldots, n\}.
\end{align}
Assume that the function $f_i:\rr^n_+\to \rr$ is continuous and satisfies $f_i(cx)<cf_i(x)$ $\forall c>1,x\in \rr^n_{++}$. 
Moreover, assume that the order-preserving property holds: $f_i(x)\leq f_i(y)$ for any $i$, for any $x,y\in \rr_+^n$ with $x_j\leq y_j$ $\forall j=1,\ldots,n$. Assume also that there exists a unique positive fixed point $x^*$ (i.e.,  $f_i(x^*)=x^*_i>0$  $\forall i$).

Then for any $x_0\in \rr_{++}^n$, the sequence $(x_t)$ determined by $x_{t+1}=f(x_t)$ converges to $x^*$.
\end{lemma}
\begin{proof} See Appendix \ref{section-stability-proofs}.\end{proof}

The proof of Lemma \ref{convergence_general} is not trivial. We make use of the approach of monotone dynamical systems \citep{Hirsch2006}.
 
Lemma \ref{convergence_general} allows us to obtain the global convergence of our economic system.\footnote{The growth factors of Theorem \ref{theorem1} depend only on population growth and on the elasticities of the two technologies, so the size of the economy affects levels but not growth rates; this is the non-scale property emphasized by \citet{EicherTurnovsky1999}. Convergence in that literature is obtained by linearizing around the balanced growth path, hence only for economies that start close to it, whereas the result below holds from any positive initial stocks.}
\begin{theorem}[Global Stability of the BGP]\label{global_stability}
Assume that $G_a>1-\delta_a,G_k>1-\delta_k, \alpha_a+\alpha_k\in (0, 1),  \gamma_a+\alpha_a\gamma_h\geq 0, \gamma_a+\alpha_a\gamma_h+\alpha_k\gamma_h \in (0,1).$

If $K_0>0,A_0>0$, then $\hat{K}_t = \frac{K_t}{G_k^t},  \hat{A}_t = \frac{A_t}{G_a^t}$ converge to $K_*,A_*$ (determined in Theorem \ref{theorem1}). Consequently, the growth rates ${K_{t+1}}/{K_t}$ and ${A_{t+1}}/{A_t}$ converge to $G_k$ and $G_a$ respectively.
\end{theorem}
\begin{proof}Denote 
\begin{subequations}
\begin{align}\label{abi_definition}
a_i&\equiv \frac{1-\delta_i}{G_i}, \quad  b_i\equiv \frac{X_i}{G_i} \text{ for  } i=k,a\\
f_1(K,A)&\equiv a_kK+b_kA^{\alpha_a}K^{\alpha_k}\\
f_2(K,A)&\equiv a_aA+b_aA^{\gamma_a+\alpha_a\gamma_h}K^{\alpha_k\gamma_h}.
\end{align}    
\end{subequations}
The system  (\ref{system-tilde}) becomes $
\hat{K}_{t+1}=f_1(\hat{K}_t,\hat{A}_t), 
 \hat{A}_{t+1}=f_2(\hat{K}_t,\hat{A}_t).$
 
 By our assumption $\alpha_a+\alpha_k\in (0, 1),  \gamma_a+\alpha_a\gamma_h\geq 0, \gamma_a+\alpha_a\gamma_h+\alpha_k\gamma_h \in (0,1)$, the functions $f_1,f_2$ are strictly concave and satisfy the order-preserving property. Moreover, they satisfy all conditions of Lemma \ref{convergence_general}. So, we can apply Lemma \ref{convergence_general} to obtain the convergence of $(\hat{K}_t)$ and $(\hat{A}_t)$.\footnote{Since  Lemma \ref{convergence_general} applies for multidimensional cases, it would be possible to generalize Theorem \ref{global_stability} to a model with multiple factors.}
\end{proof}


\paragraph{Monotonicity of $\hat{K_t}$ and $\hat{A_t}$.}  
In Theorem \ref{global_stability}, we have proved that $\hat{K_t}$ and $\hat{A_t}$ converge. However, with additional conditions on initial conditions, we can go further by demonstrating that they are increasing or decreasing over time, as shown by the following result. 
\begin{proposition}[\emph{Monotone normalized trajectories}]
    \label{monotone-remark}  Assume that $a_i\in (0,1), \gamma_a+\alpha_a\gamma_h\in (0,1)$.
\begin{enumerate}
    \item 
If the initial capital and knowledge stocks satisfy
 \begin{align}\label{system-tilde-date0}
b_k\hat{A}_0^{\alpha_a}\hat{K}_0^{\alpha_k-1}\geq 1-a_k \text{ and }
b_a\hat{A}_0^{\gamma_a+\alpha_a\gamma_h-1}\hat{K}_0^{\alpha_k\gamma_h}\geq 1-a_a.
\end{align}
then the sequences $\hat{K_t}$ and $\hat{A_t}$ are increasing in time.

\item If the initial capital and knowledge stocks satisfy
 \begin{align}\label{system-tilde-date1}
b_k\hat{A}_0^{\alpha_a}\hat{K}_0^{\alpha_k-1}< 1-a_k \text{ and }
b_a\hat{A}_0^{\gamma_a+\alpha_a\gamma_h-1}\hat{K}_0^{\alpha_k\gamma_h}< 1-a_a.
\end{align}
then the sequences $\hat{K_t}$ and $\hat{A_t}$ are decreasing in time.
\end{enumerate}
\end{proposition}
\begin{proof} See Appendix \ref{section-stability-proofs}.\end{proof}
\subsubsection{Global stability under general production functions}
Theorem \ref{global_stability} shows  global stability under Cobb--Douglas functions. However, thanks to Lemma \ref{convergence_general}, we can actually prove the global stability of the system under general functions $F$ and $F_a$, which satisfy mild conditions. The output equals $Y_t=F(A_t,K_t,\lambda_yL_t)$ and the economic system is now described  by
    \begin{subequations}\label{SystemGeneral}
\begin{align}
S_t&=sY_t, \quad L_{t}=L_0G_n^t, \quad L_{yt}=\lambda_yL_t, \quad  L_{at}=\lambda_aL_t, \quad \lambda_a+\lambda_y=1\\
H_{t+1}&=\gamma s F(A_t,K_t,\lambda_yL_t)\\
\label{Kgeneral} K_{t+1}&= (1-\delta_k)K_t+(1-\gamma)s F(A_t,K_t,\lambda_yL_t) \\
\label{Ageneral}A_{t+1}&=(1-\delta_a)A_t+ F_a(A_t, H_{t+1},\lambda_aL_{t+1})
\end{align}
\end{subequations}
where  $K_0, A_0,L_0$  are given and are positive.
%
Define the functions $\hat{F}$ and $\hat{F}_a$ by $\hat{F}(\hat{A},\hat{K})\equiv F(\hat{A},\hat{K},\lambda_y)$ and 
$\hat{F}_a(\hat{A},\hat{K})=F_a(\hat{A},\gamma s\hat{F}(\hat{A},\hat{K}),\lambda_aG_n)$.

\begin{assum}\label{FF_a}

\begin{enumerate}
  \item \label{FF_a1} The functions $\hat{F}$ and $\hat{F}_a$ are strictly increasing in each component, continuously differentiable, strictly concave and $\hat{F}(0,0)=\hat{F}_a(0,0)=0$.

\item\label{FF_a2}  There exist $\psi_a,\psi_k>0$ such that, $\forall \lambda,A,K,H,L\geq 0,$
\begin{subequations}
\begin{align}
F(\lambda^{\psi_a}A,\lambda^{\psi_k}K,\lambda L)&=\lambda^{\psi_k}F(A,K,L),\\
F_a(\lambda^{\psi_a}A,\lambda^{\psi_k}H,\lambda L)&=\lambda^{\psi_a}F_a(A,H,L).
\end{align}
\end{subequations}
\item \label{FF_a3} There exists a unique vector $(\hat{K}_*,\hat{A}_*)\in \rr_{++}^2$ such that
\begin{align}\label{vectorAK}
G_k\hat{K}_*&=(1-\delta_k)\hat{K}_*+(1-\gamma)s\hat{F}(\hat{A}_*,\hat{K}_*), &
G_a\hat{A}_*&=(1-\delta_a)\hat{A}_*+\hat{F}_a(\hat{A}_*,\hat{K}_*).
\end{align}
\end{enumerate}
\end{assum}
We have the following result regarding the existence and stability of BGP.
\begin{theorem}[Global stability under general technologies]\label{global_stability_general} 
Consider the system (\ref{SystemGeneral}) and let Assumption \ref{FF_a} be satisfied. Then, there exists a BGP, whose growth rates are given by $G_y=G_k=G_n^{\psi_k}$ and $G_a=G_n^{\psi_a}$. Moreover, this BGP is stable in the sense that: when $K_0>0,A_0>0$, the sequences ${A_t}/{G_a^t},{K_t}/{G_k^t}$ converge to $L_0^{\psi_a}\hat{A}_*,L_0^{\psi_k}\hat{K}_*$.

\end{theorem}
\begin{proof} See Appendix \ref{section-stability-proofs}.\end{proof}

Theorem \ref{global_stability_general} is a generalized version of Theorem \ref{global_stability}. Indeed, in Theorem \ref{global_stability}, $\psi_k$ and $\psi_a$ are given by \eqref{Ga*} while $\hat{A}_*=A_*$ and $\hat{K}_*=K_*$ are determined by \eqref{formulaA_ss} and \eqref{formulaK_ss}.

\subsection{Instability of the economic system: the influence of the difficulty of innovation}
\label{section-instability}
In the accumulation of knowledge \eqref{AGamma}, the parameter $\gamma$ and the effective elasticity $\Gamma\equiv \gamma_a+\alpha_a\gamma_h$ play an important role.  We will demonstrate that when $\gamma_a$ is sufficiently low, our economic system can become unstable. 
Recall that
\[
    f_1(K, A) = a_k K + b_k A^{\alpha_a} K^{\alpha_k}, \qquad
    f_2(K, A) = a_a A + b_a A^{\gamma_a + \alpha_a \gamma_h} K^{\alpha_k \gamma_h}.
\]

Our dynamical system (\ref{system-tilde}) becomes
\begin{align}\label{system-KA}
\begin{bmatrix}
    \hat{K}_{t+1} \\
    \hat{A}_{t+1}
\end{bmatrix}
=\begin{bmatrix}
    f_1(\hat{K}_t, \hat{A}_t) \\
    f_2(\hat{K}_t, \hat{A}_t)
\end{bmatrix}
=
\begin{bmatrix}
    a_k \hat{K}_t + b_k \hat{A}_t^{\alpha_a} \hat{K}_t^{\alpha_k} \\
    a_a \hat{A}_t + b_a \hat{A}_t^{\gamma_a + \alpha_a \gamma_h} \hat{K}_t^{\alpha_k \gamma_h}
\end{bmatrix}.
\end{align}


\begin{proposition}[\emph{Instability of the Steady State}]\label{unstable}Consider the steady state $(K_*, A_*)$ of the system (\ref{system-tilde}) and denote  $J_*=J(K_*,A_*)$ the Jacobian matrix at $(K_*, A_*)$. 

The steady state $(K_*, A_*)$ is unstable if  $\phi_a\equiv(1-\gamma_a)(1-\alpha_k)-\alpha_a\gamma_h<0$ or $\big|det(J_*)\big|>1$, where \begin{align*}det(J_*)=    (a_k + \alpha_k(1 - a_k))(a_a + (\gamma_a + \alpha_a \gamma_h)(1 - a_a)) - \alpha_a\gamma_h (1 - a_k) \alpha_k  (1 - a_a).
\end{align*}
This leads to two particular cases:
\begin{enumerate}
    \item\label{gamma_a_low}
If  $det(J_*)<-1$ (which happens when $\gamma_a$ is small enough), then the steady state $(K_*, A_*)$ is unstable.

\item\label{gamma_a_hight} If $det(J_*)>1$ or $\phi_a<0$ (which happens when $\gamma_a$ is high enough), then the steady state $(K_*, A_*)$ is unstable.
\end{enumerate}
\end{proposition} 
\begin{proof}    See Appendix \ref{section-instability-proofs}.\end{proof}


Particular case \ref{gamma_a_low} in Proposition \ref{unstable} shows that a sufficiently low value of $\gamma_a$ can make the steady state unstable. Economically, a low $\gamma_a$ means that the production of new ideas responds weakly, or even negatively, to the existing stock of knowledge. In this case, as the knowledge stock increases, the generation of new ideas may fall short with the depreciation or obsolescence of existing knowledge. If investment in the inputs that support knowledge accumulation is insufficient to compensate for this decline, the effective knowledge stock can fall, generating a destabilizing feedback. Thus, sufficiently weak knowledge spillovers in the innovation process can limit convergence to the steady state and lead instead to dynamic instability.

Point \ref{gamma_a_hight} highlights a different source of instability. When $\gamma_a$ is sufficiently high, the production of new ideas becomes highly responsive to the existing stock of knowledge. A small increase in the knowledge stock can then generate a disproportionately large increase in new ideas, which further raises the knowledge stock in subsequent periods. This strong positive feedback can cause the economy to exceed its steady state rather than converge toward it, making the steady state unstable.

Middle values of $\gamma_a$ will be discussed in the following result, which also complements Theorem \ref{global_stability}.

\begin{proposition}[Asymptotic stability]
\label{asymptotically-stable}
The steady state $(K_*, A_*)$ of the system (\ref{system-tilde})  is asymptotically stable if $a_k,a_a\in (0,1)$ and
\begin{subequations}
\begin{align}
\label{tracepositive}   
  - \frac{a_k + \alpha_k (1 - a_k) + a_a}{1-a_a}< \gamma_a + \alpha_a \gamma_h<1,\\
 \label{differencepositive}     (1 - \alpha_k)(1 - \gamma_a) - \alpha_a \gamma_h>0.
\end{align}
\end{subequations}

\end{proposition}
\begin{proof}    See Appendix \ref{section-instability-proofs}.\end{proof}


In Theorem \ref{global_stability}, we impose $\gamma_a+\alpha_a\gamma_h\geq 0$ to establish global asymptotic stability of $(K_*,A_*)$. Proposition \ref{asymptotically-stable} relies on weaker conditions and therefore naturally yields a weaker result: it establishes asymptotic stability of $(K_*,A_*)$. 

Importantly, Proposition \ref{asymptotically-stable} allows for $\gamma_a+\alpha_a\gamma_h<0$, provided that this term is not too low, in the sense that condition (\ref{tracepositive}) is satisfied. This possibility may occur even when $\gamma_a$ is negative (as $\alpha_a$ and $\gamma_h$ are positive) so that the production of new ideas decreases with the existing stock of knowledge. This captures declining idea productivity: as the stock of knowledge expands, generating additional ideas becomes increasingly difficult. Such a mechanism is consistent with the empirical evidence on declining research productivity reported by \citet{Bloom2020} and with the semi-endogenous growth frameworks developed by \citet{JonesAER2022} and \citet{Jonesreview}.

\subsection{Innovation cycle} 


Proposition \ref{unstable} has shown that the difficulty of finding innovations may make the economic system unstable. We now argue that it may also generate an innovation cycle.  
Let us consider the system (\ref{system-KA}) and provide a formal definition of cycle.

\begin{definition}
\begin{enumerate}
    \item We say that there exists a two-period innovation cycle if there is a path $(\hat{K}_t, \hat{A}_t)_t$ satisfying the system (\ref{system-KA}) and
$(\hat{A}_{2t},\hat{A}_{2t+1})=(\hat{A}_{0},\hat{A}_{1})$ for any $t$ with $\hat{A}_{0},\hat{A}_{1}>0$, $\hat{A}_{0}\not=\hat{A}_{1}$.
\item Given a positive integer $N\geq 2$, we say that there exists an N-period innovation cycle if there exists a path $(\hat{K}_t, \hat{A}_t)_t$ satisfying the system (\ref{system-KA}) and
$\hat{A}_{t}=\hat{A}_{N+t}$ for all $t$ with $\hat{A}_{i}>0$, $\hat{A}_{i}\not=\hat{A}_{j}$ for all $i,j\in \{0,1,\ldots,N-1\}$ with $i\not=j$.
\end{enumerate}
\end{definition}

For the sake of tractability, we focus on the case $\gamma_h=0$. This assumption implies that R\&D expenditure does not directly affect the production of new ideas, and the production of new ideas may nevertheless depend on the existing stock of knowledge and researchers who make use of existing knowledge. In this case, we have 
\begin{align}\label{map_fa}
      \hat{A}_{t+1}= f_a(\hat{A}_t), \text{ where }f_a(x)\equiv a_a x + b_a x^{\gamma_a} \text{ } \forall x.
\end{align}

\begin{proposition}[Innovation Cycles]\label{cycle}Consider the system (\ref{system-KA}). Assume that $\gamma_h=0$ and $\gamma_a<0$.

\begin{enumerate}
\item \label{point_2period} 

If $-1\leq \gamma_a <0$, there does not exist a two-period innovation cycle. 

However, if 
\begin{align}\label{gamma_acondition}
    \gamma_a+1+(1-\gamma_a)\frac{1-\delta_a}{G_n^{\frac{\gamma_\ell}{1-\gamma_a}}}<0,
\end{align}
then there exists a two-period innovation cycle $(\hat{A}_{0},\hat{A}_{1})$. Note that the condition (\ref{gamma_acondition}) holds if $\gamma_a$ is low enough.

\item \label{point_Nperiod} Assume that $\gamma_a$ is negative and also satisfies
\begin{align}\frac{a^2}{(a+1)(1+2a+2a^2)}>\big(\frac{a+1}{2a}\big)^{\gamma_a}, \text{ where } a\equiv \frac{1-\delta_a}{G_n^{\frac{\gamma_\ell}{1-\gamma_a}}}.\label{key_gamma_a}
\end{align}
Then, for every integer $N\geq 2$, there exists an N-period innovation cycle  $(\hat{A}_0,\hat{A}_1,\ldots, \hat{A}_{N-1})$. 
\end{enumerate}
\end{proposition}
\begin{proof}    See Appendix \ref{section-instability-proofs}.\end{proof}

The key point in  Proposition~\ref{cycle} is that a two-period innovation cycle may arise when $\gamma_a$ is sufficiently negative, i.e., the elasticity of production of ideas with respect to existing knowledge $\gamma_a$ is sufficiently negative. In this case, the economic system may fail to converge to a balanced growth path and instead follow a $N$-period cycle, where the length $N$ can be arbitrary.  To prove Proposition \ref{cycle}'s point \ref{point_Nperiod}, we first prove, by using \eqref{key_gamma_a}, the existence of a three-period cycle. Then, by the Li-Yorke Theorem \citep{LiYorke1975}, we have a $N-$period cycle.

Observe that conditions (\ref{gamma_acondition}) and (\ref{key_gamma_a}) for having innovation cycles can be satisfied if $\gamma_a$ is low and/or the depreciation rate $\delta_a$ is high. For example, when $G_n=1.1$, $\gamma_\ell=0.4$ and $\delta_a=0.5$, condition (\ref{gamma_acondition}) holds once $\gamma_a$ falls below $-2.96$, and condition (\ref{key_gamma_a}) once $\gamma_a$ falls below $-6.64$.\footnote{Recall that it would be more difficult to have a 3-period cycle than to have a 2-period cycle. Condition (\ref{key_gamma_a}) is only sufficient: three-period cycles can be located numerically for higher values of $\gamma_a$ as well. At $\gamma_a=-6$, for instance, the map $f_a$ admits the three-period cycle $\hat{A}_0\approx 1.345619$, $\hat{A}_1=f_a(\hat{A}_0)\approx 0.957586$, $\hat{A}_2=f_a(\hat{A}_1)\approx 2.696985$, $\hat{A}_3=\hat{A}_0$, even though (\ref{key_gamma_a}) does not hold there.}

The cycles in Proposition~\ref{cycle} arise from delayed negative feedback. When knowledge is above its balanced-growth level, the negative elasticity $\gamma_a$ reduces idea production, so knowledge subsequently falls relative to trend. When knowledge is below trend, idea production rises and knowledge recovers. Because this adjustment operates with a one-period lag, a sufficiently strong response can generate overshooting and alternating phases of high and low innovation rather than convergence. Knowledge depreciation reinforces this mechanism by reducing the persistence of past knowledge; it should be interpreted as economic obsolescence rather than the literal disappearance of ideas.

\begin{remark}[Cycles with $\gamma_h>0$]
\label{remark:cycles-gammah-positive}
The assumption $\gamma_h=0$ in Proposition~\ref{cycle} makes the knowledge dynamics one-dimensional and allows analytical results for cycles of arbitrary period. Remark \ref{remark:positive-rd-cycles} in Section~\ref{subsec:cycle-sim} provides an example with $\gamma_h=0.4$ in which the two-dimensional detrended system admits a locally attracting two-period cycle. Thus, cycles may also arise when R\&D expenditure directly affects innovation and capital feeds back into knowledge accumulation.
\end{remark}


It should be noted that innovation cycles in our framework are  endogenous and fully deterministic. They arise without changes in behavioral parameters or extrinsic shocks.  Our mechanism of the endogenous cycle contrasts with \citet{he2022}, where cycles arise from the interaction of learning-by-doing and fishing-out effects, together with step-wise technology upgrades. In our model, cycles are deterministic orbits of the knowledge map and can have any period $N$, whereas in \citet{he2022}'s model, they appear as inverted-U growth pulses of varying lengths. Together, these approaches point to different but complementary ways in which innovation can generate cyclical growth.

Our mechanism is different from and appears to be simpler than that of \citet{walde05}, who uses the portfolio approach (he considers that R\&D is an uncertain activity that is modeled by the Poisson process) to explain endogenous fluctuations. It also differs from that of \citet{EvansHonkapohjaRomer1998}. In their model, complementarity between capital goods generates two rational expectations growth paths, a fast one and a slow one, and an extrinsic random variable coordinates self-fulfilling shifts of expectations between them. Their cycles therefore rely on multiple equilibria and extrinsic uncertainty. In contrast, the periodic paths coexist with the unique BGP and occur away from it (Theorem \ref{theorem1}) in a fully deterministic and backward-looking economy, driven by the difficulty of producing new ideas when the stock of knowledge is large. The source of instability is instead the negative effect of the existing stock of knowledge on the production of new ideas.


\section{Numerical simulations}
\label{section5}
\subsection{Parameters}\label{params}
Table~\ref{tab:baseline-params} lists the baseline parameterization and, where one exists, the source for each value; one period is one year. {Three choices need comment. We set $\alpha_\ell=0.6$ rather than $0.7$ so that the three exponents sum to one ($\alpha_k+\alpha_\ell+\alpha_a=1$), which treats $A$ as an accumulated factor alongside capital and labor, as in \citet[Section~II.A, p.~416]{MankiwRomerWeil1992}; the benchmark $\alpha_k+\alpha_\ell=1$ is examined below. We set $\gamma_a=-0.2$, an order of magnitude smaller in absolute value than the aggregate evidence on declining research productivity \citep[Table~7, p.~1134]{Bloom2020};\footnote{{\citet[p.~127]{Jonesreview} writes ideas as $\dot A_t=R_t^{\lambda}A_t^{1-\beta}$; see also \citet[footnote~25, p.~1135]{Bloom2020}. The exponent $1-\beta$ corresponds to our parameter $\gamma_a$, and it is read there exactly as here: when it is negative, past discoveries ``make it harder to generate new ideas'' \citep[p.~127]{Jonesreview}. For the aggregate US economy $\beta=3.1$ \citep[Table~7, p.~1134]{Bloom2020}, that is, $\gamma_a\approx-2$ \citep[p.~131]{Jonesreview}. Both that value and our baseline lie inside the region of Proposition \ref{asymptotically-stable}; the cycle configuration of Section \ref{subsec:cycle-sim} is far more extreme, and we return to this gap there.}} 
the implied $\Gamma=-0.16$ satisfies Proposition \ref{asymptotically-stable}, so the baseline economy is locally stable. Finally, $s\gamma=0.06$ and $\lambda_y=0.9$ deliberately exceed observed R\&D intensities (about $2\%$--$3\%$ of GDP in the OECD) and research-employment shares: both make the R\&D channel visible in figures drawn on a common scale, neither affects long-run growth, and Section \ref{subsec:policy-sim} varies both.}


{\paragraph{Robustness to the returns-to-scale configuration.}
No result in this paper restricts aggregate returns to scale, and the choice can be settled analytically. Under the benchmark $\alpha_k+\alpha_\ell=1$, constant returns apply to the rival inputs alone, non-rivalry then delivering increasing returns to ideas and rival factors taken together \citep[p.~127]{Jonesreview}; it gives $\psi_k=1+(\alpha_a/\alpha_\ell)\psi_a$, which exceeds $1$ when $\alpha_a>0$ and 
\begin{align}\label{assum_additional}\phi_a\equiv (1-\gamma_a)(1-\alpha_k)-\alpha_a\gamma_h>0,\end{align} so per-capita output grows whatever the strength of the knowledge externality. Under our baseline $\psi_k$ is instead a weighted average of $1$ and $\psi_a$, hence $\psi_k=0.95$ and per-capita output falls by about $0.10\%$ per period, against $\psi_k=1.10$ and $+0.20\%$ under $\alpha_\ell=0.7$, where $\psi_a$ also rises from $0.65$ to $0.70$ and knowledge grows at $1.40\%$ instead of $1.30\%$. The dynamics are unaffected: neither $\phi_a=0.80$ nor $\Gamma=-0.16$ depends on $\alpha_\ell$, and the cycles of Proposition \ref{cycle}, derived under $\gamma_h=0$, are independent of $(\alpha_a,\alpha_k,\alpha_\ell)$ altogether. The configuration shifts the \emph{level} of long-run growth, not existence, uniqueness, stability, or the possibility of cycles.}

\begin{table}[H]
  \centering
  \footnotesize
  \setlength{\tabcolsep}{4pt}
  \renewcommand{\arraystretch}{1.12}
  \caption{Parameter values used in the simulation.}\label{tab:baseline-params}
  \begin{tabularx}{\textwidth}{@{}>{\raggedright\arraybackslash\leavevmode}p{3.8cm}
                                 c c
                                 >{\raggedright\arraybackslash\leavevmode}X@{}}
    \toprule\toprule
    \textbf{Quantity} & \textbf{Symbol} & \textbf{Value} & {\textbf{Source / justification}} \\
    \midrule
    \addlinespace[1pt]
    \multicolumn{4}{@{}l}{\textit{Production and allocation}}\\
    \addlinespace[1pt]
    Capital share            & $\alpha_k$    & $0.3$  &  {Capital's share in income is ``roughly one third'' \citep[p.~410]{MankiwRomerWeil1992}}\\
    Labor share              & $\alpha_\ell$ & $0.6$  &  {Constant returns to $(K,H,L)$ jointly, with $A$ in the role of $H$ \citep[Section~II.A, p.~416]{MankiwRomerWeil1992}; $\alpha_\ell=0.7$ discussed below}\\
    Knowledge share          & $\alpha_a$    & $0.1$  &  {Modest; condition \eqref{assum_additional} then holds with room, $(1-\gamma_a)(1-\alpha_k)=0.84>\alpha_a\gamma_h=0.04$}\\
    Saving rate              & $s$           & $0.3$  &  {Gross investment share of real GDP, measured as in \citet[p.~412]{MankiwRomerWeil1992}; at the upper end of observed averages}\\
    R\&D share of savings    & $\gamma$      & $0.2$  &  {Implies $s\gamma=0.06$, above observed R\&D intensities (see text)}\\
    \addlinespace[3pt]
    \multicolumn{4}{@{}l}{\textit{Depreciation and innovation}}\\
    \addlinespace[1pt]
    Capital depreciation     & $\delta_k$    & $0.05$ &  {Conventional value; \citet[p.~413]{MankiwRomerWeil1992} assume the \emph{sum} $g+\delta=0.05$, so this is an upper bound for $\delta_k$ alone}\\
    Knowledge depreciation   & $\delta_a$    & $0.02$ &  {Normalization; affects levels and the transition only (Theorem \ref{theorem1})}\\
    Innovation efficiency    & $\mu_a$       & $0.05$ &  {Normalization; affects levels and the transition only (Theorem \ref{theorem1})}\\
    Knowledge elasticity     & $\gamma_a$    & $-0.2$ &  {Declining research productivity: the aggregate $\beta=3.1$ \citep[Table~7, p.~1134]{Bloom2020} implies $\gamma_a\approx-2$, so $-0.2$ is conservative (see text)}\\
    R\&D-expenditure elasticity & $\gamma_h$ & $0.4$  &  {Decreasing returns to research inputs, $\gamma_h+\gamma_\ell=0.8<1$, as in the duplication externality of \citet{Jones1995}}\\
    R\&D labor elasticity    & $\gamma_\ell$ & $0.4$  &  {Same as $\gamma_h$}\\
    \addlinespace[3pt]
    \multicolumn{4}{@{}l}{\textit{Labor allocation and demographics}}\\
    \addlinespace[1pt]
    Production labor share   & $\lambda_y$   & $0.9$  &  {Assigns $10\%$ of labor to research, well above observed research-employment shares (see text)}\\
    Population growth        & $G_n$         & $1.02$ &  {Round value for long-run population growth; $n$ measured as the growth of the working-age population \citep[p.~412]{MankiwRomerWeil1992}}\\
    \midrule
    \addlinespace[1pt]
    \multicolumn{4}{@{}l}{ {\textit{Implied values}}}\\
    \addlinespace[1pt]
     {Total elasticity}    &  {$\Gamma$}        &  {$-0.16$}       &  {$\Gamma=\gamma_a+\alpha_a\gamma_h$; inside the region of Proposition \ref{asymptotically-stable}}\\
     {Growth elasticities} &  {$\psi_a,\psi_k$} &  {$0.65,\,0.95$} &  {Theorem \ref{theorem1}: $G_a=G_n^{\psi_a}$, $G_y=G_k=G_n^{\psi_k}$}\\
    \bottomrule
  \end{tabularx}
  \vspace{2pt}
  {\footnotesize\begin{minipage}{\textwidth}
  \emph{Notes.}  {Initial conditions are $K_0=10$ and $A_0=5$, which affect
  the transition only, and $L_0=50$, which enters $X_k$ and $X_a$ in
  \eqref{xkxa} and therefore scales the steady-state levels $(A_*,K_*,Y_*)$
  without affecting any growth rate. The last block reports quantities implied
  by the calibration rather than free parameters.}
  \end{minipage}}
\end{table}

\subsection{Baseline}\label{subsec:baseline-sim}
Figure~\ref{fig:baseline} reports the baseline over 150 periods, from $K_0=10$,
$A_0=5$, $L_0=50$. In panel~(a) capital, output, labor and the R\&D flow
$H_{t+1}=\gamma sY_t$ of \eqref{2g} grow throughout, while the knowledge stock
first falls for five periods -- the research flow does not yet offset
obsolescence at $\delta_a=0.02$ -- and then grows, but always more slowly than
output, since with $\gamma_a=-0.2$ a larger stock lowers the output of new ideas
for given research inputs.\footnote{The detrended stock $\hat A_t$ is therefore
not monotone, consistent with Proposition~\ref{monotone-remark}, which requires
$\Gamma=\gamma_a+\alpha_a\gamma_h\in(0,1)$ and so does not apply here, where
$\Gamma=-0.16$.} Output and labor end the sample almost superimposed: with
$\psi_k=0.95<1$, output per capita peaks at period~$58$ and then declines
towards $G_n^{\psi_k-1}-1\approx-0.10\%$ per period (Section~\ref{params}).

Panel~(b) shows the growth rates converging to the values of
Theorem~\ref{theorem1}, $G_n^{\psi_k}-1\approx1.90\%$ for output and capital and
$G_n^{\psi_a}-1\approx1.30\%$ for knowledge, with $g_{A,t}$ below $g_{Y,t}$ at
every date. Since Proposition~\ref{asymptotically-stable} guarantees stability
only near the BGP, the convergence shown here - from initial rates of $+50\%$
for capital and $-0.4\%$ for knowledge, far outside the panel - is a numerical
result rather than an implication of the theorem.

\begin{figure}[H]
  \centering
  \includegraphics[width=.92\linewidth]{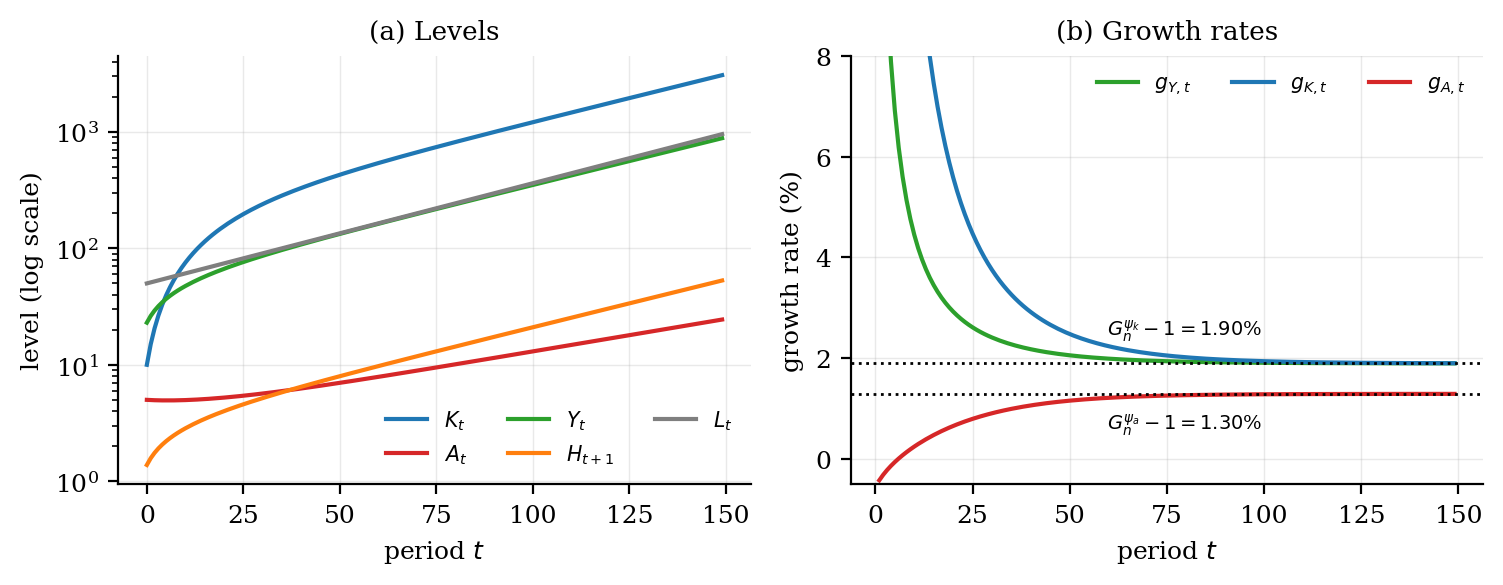}
  \caption{Baseline dynamics ($K_0=10$, $A_0=5$, $L_0=50$). (a) Levels of
  $K_t$, $A_t$, $Y_t$, $H_{t+1}$, $L_t$ (log scale); (b) growth rates
  $g_{Y,t}$, $g_{K,t}$, $g_{A,t}$, with the balanced-growth values of
  Theorem~\ref{theorem1} as dotted lines; the first periods of $g_{K,t}$ and
  $g_{Y,t}$ lie above the vertical range shown.}
  \label{fig:baseline}
\end{figure}

\subsection{Policy experiments}\label{subsec:policy-sim}
We now shock one parameter at a time at $t=20$ and compare the resulting paths of $g_{A,t}$ and $g_{Y,t}$ with the baseline. Figure~\ref{fig:policy} reports the deviations from the baseline path and Table~\ref{tab:policy} reports the size of the effects. The date of the intervention matters for the transition paths but not for the long-run columns. The experiments are permanent unless stated otherwise; decreases mirror the increases reported here, and temporary interventions (the parameter is changed for ten periods and then restored) are illustrated in panel~(c).

\begin{figure}[H]
  \centering
  \includegraphics[width=.98\linewidth]{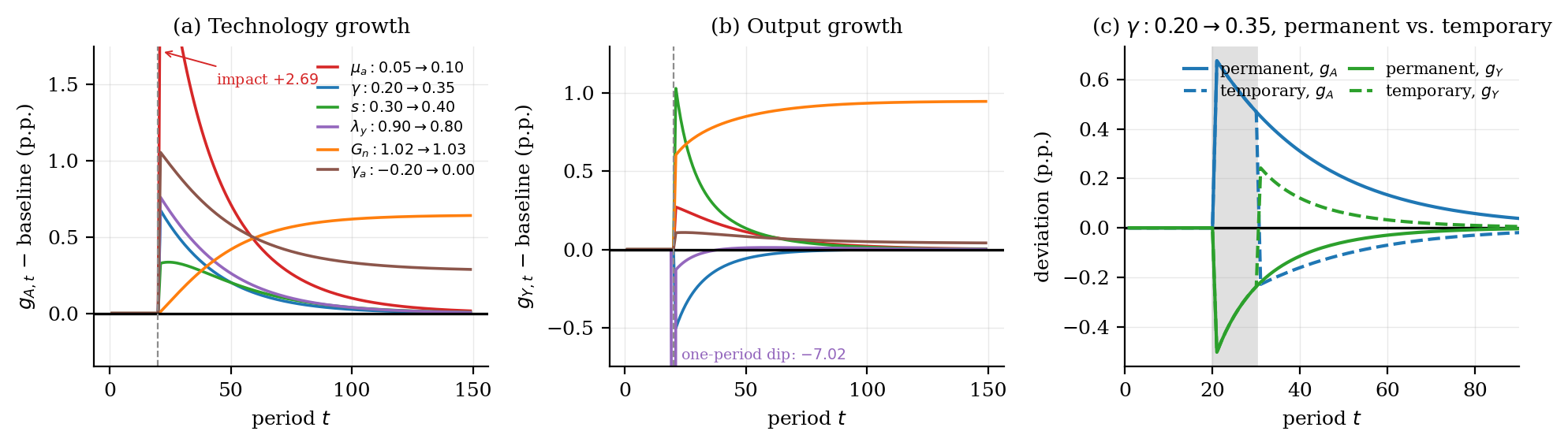}
  \caption{Policy experiments: deviations from the baseline growth path, in   percentage points, for permanent changes introduced at $t=20$ (dashed
  vertical line). (a) Technology growth (the impact effect of the $\mu_a$   increase, $+2.69$, is truncated); (b) output growth (the one-period level
  effect of the $\lambda_y$ reallocation is truncated); (c) permanent   versus temporary increase in $\gamma$, the shaded band being the intervention
  window.}
  \label{fig:policy}
\end{figure}

\begin{table}[H]
  \centering
  \footnotesize
  \setlength{\tabcolsep}{5pt}
  \renewcommand{\arraystretch}{1.12}
  \caption{Policy experiments: deviations from the baseline growth rates, in
  percentage points.}\label{tab:policy}
  \begin{tabularx}{\textwidth}{@{}>{\raggedright\arraybackslash\leavevmode}p{3.6cm}
        >{\leavevmode}X
        >{\centering\arraybackslash}p{1.25cm}>{\centering\arraybackslash}p{1.25cm}
        >{\centering\arraybackslash}p{1.25cm}>{\centering\arraybackslash}p{1.25cm}@{}}
    \toprule\toprule
    & & \multicolumn{2}{c}{\textbf{Short run}} & \multicolumn{2}{c}{\textbf{Long run}}\\
    \cmidrule(lr){3-4}\cmidrule(lr){5-6}
    \textbf{Experiment} & \textbf{Channel} & $\Delta g_A$ & $\Delta g_Y$ & $\Delta g_A$ & $\Delta g_Y$\\
    \midrule
    \addlinespace[1pt]
    $\mu_a:0.05\to0.10$     & Research efficiency           & $+2.69$ & $+0.27$ & $0$ & $0$\\
    $\gamma:0.20\to0.35$    & Savings reallocated to R\&D   & $+0.67$ & $-0.50$ & $0$ & $0$\\
    $s:0.30\to0.40$         & Saving rate                   & $+0.34$ & $+1.03$ & $0$ & $0$\\
    $\lambda_y:0.90\to0.80$ & Labor reallocated to R\&D     & $+0.76$ & $-7.02$ & $0$ & $0$\\
    \addlinespace[3pt]
    $G_n:1.02\to1.03$       & Demography                    & ---     & ---     & $+0.64$ & $+0.95$\\
    $\gamma_a:-0.20\to0.00$ & Knowledge elasticity          & $+1.05$ & $+0.11$ & $+0.28$ & $+0.04$\\
    \bottomrule
  \end{tabularx}

  \vspace{2pt}
  {\footnotesize\begin{minipage}{\textwidth}
  \emph{Notes.} All experiments are permanent and introduced at $t=20$. The short-run column is the largest deviation from the baseline growth rate during
  the first ten periods after the intervention; the long-run column is the difference between the balanced-growth rates of Theorem~\ref{theorem1}. For
  $G_n$ the deviation rises monotonically towards its long-run value, so there is no peak to report; for $\gamma_a$ the deviation overshoots, reaching
  $+1.05$ on impact before settling at $+0.28$. The $\lambda_y$ entry for $\Delta  g_Y$ is a one-period level effect: moving $10\%$ of the labor force out of   production lowers output once by $6.8\%$, and from the next period on the   deviation stays below $0.15$. Decreases mirror the increases reported here, and temporary versions leave the long-run columns unchanged.
  \end{minipage}}
\end{table}

The experiments confirm Theorem~\ref{theorem1} one by one. Raising research efficiency $\mu_a$, the R\&D share of savings $\gamma$, the saving rate $s$ or
the research share of labor $1-\lambda_y$ moves both growth rates on impact, sometimes by a lot, but every path returns to the baseline rates: these
parameters shift the transition and the steady state levels $(A_*,K_*,Y_*)$, not long-run growth. Two of them show the short-run trade-off: moving savings
into R\&D ($\gamma$) or labor into research ($\lambda_y$) raises $g_{A,t}$ and lowers $g_{Y,t}$, because physical investment and production labor fall. Only two experiments move the asymptotes. Population growth $G_n$ raises both rates permanently, and the knowledge elasticity $\gamma_a$ enters long-run growth
through $\phi_a=(1-\gamma_a)(1-\alpha_k)-\alpha_a\gamma_h$ and the stability conditions of Section~\ref{section-instability} through
$\Gamma=\gamma_a+\alpha_a\gamma_h$. Panel~(c) contrasts a permanent and a temporary increase in $\gamma$: the temporary one gives a hump in $g_{A,t}$ inside the window, a mirror dip after it, and no lasting effect.

\subsection{Innovation cycles}\label{subsec:cycle-sim}
We illustrate Proposition~\ref{cycle} under $\gamma_h=0, G_n=1.1, \gamma_\ell=0.4,   \delta_a=0.5.$ Using the normalized variable $y_t\equiv \hat A_t/A_*$, the knowledge dynamics
become $y_{t+1}=a_a y_t+(1-a_a)y_t^{\gamma_a}.$

For these parameters, the steady state loses stability at
$\gamma_a\approx-2.96$, corresponding to condition~\eqref{gamma_acondition}.
Figure~\ref{fig:cycles}(a) compares convergence at $\gamma_a=-2$ with the
two-period cycle obtained at $\gamma_a=-3.5$. The stronger sufficient
condition~\eqref{key_gamma_a} for a three-period cycle holds from approximately
$\gamma_a=-6.64$; panel~(b) illustrates such a cycle at $\gamma_a=-7$.
Thus, as the negative effect of existing knowledge on idea production becomes
sufficiently strong, convergence to the BGP may give way to endogenous
innovation cycles.

The numerical thresholds also illustrate the interaction between the knowledge
elasticity and knowledge depreciation. For example, condition
\eqref{gamma_acondition} holds at $\delta_a=0.5$ when
$\gamma_a\lesssim-2.96$, whereas with the baseline
$\delta_a=0.02$ it requires a much more negative value of $\gamma_a$.
Accordingly, the cyclical configurations considered here should be viewed as
illustrations of fast-obsolescence environments rather than as a calibration of
aggregate knowledge dynamics.

\begin{figure}[H]
  \centering
  \includegraphics[width=0.9\linewidth]{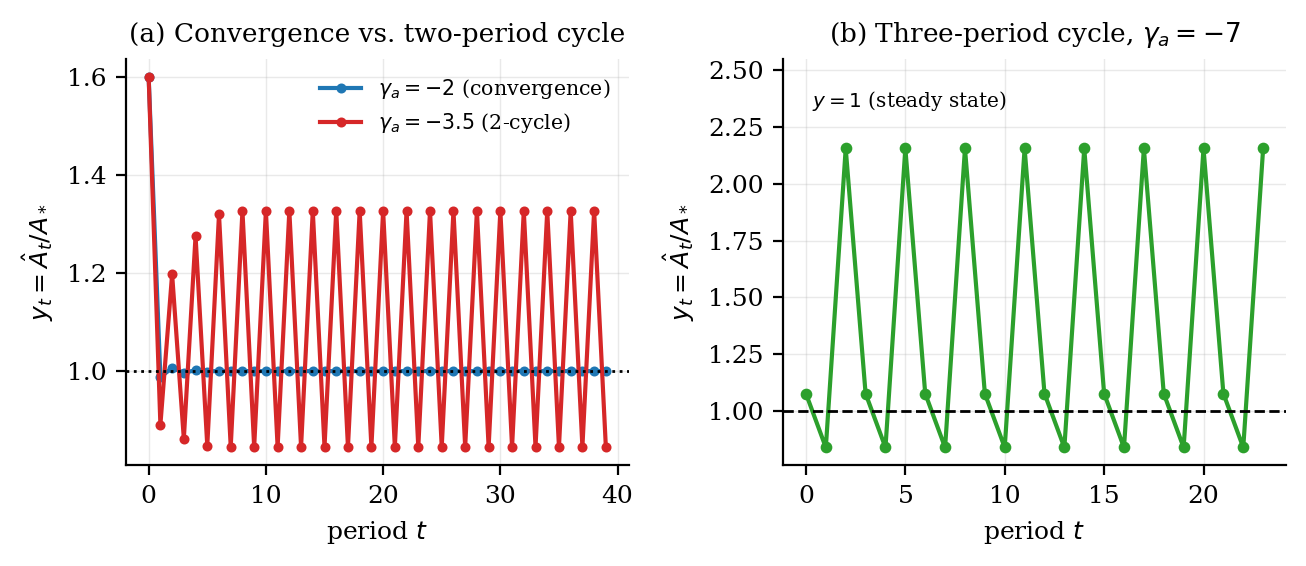}
  \caption{{Innovation cycles (Proposition \ref{cycle}; $\gamma_h=0$,
  $G_n=1.1$, $\gamma_\ell=0.4$, $\delta_a=0.5$). 
  (a) Convergence ($\gamma_a=-2$) versus two-period cycle ($\gamma_a=-3.5$).
  (b) A three-period cycle at $\gamma_a=-7$.}}
  \label{fig:cycles}
\end{figure}

\begin{remark}[Cycles with $\gamma_h>0$]
\label{remark:positive-rd-cycles}
The restriction $\gamma_h=0$ in Proposition~\ref{cycle} is imposed for analytical tractability, but is not necessary for the existence of a two-period cycle. For example, retaining the parameters of Table~\ref{tab:baseline-params}, except for
$\gamma_h=0.4, \delta_a=0.5, G_n=1.1, \gamma_a=-3.168718$, the detrended system has a locally attracting two-period cycle. More precisely,
along this cycle,
$$
\Big(
\frac{\hat K_{2t}}{K_*},
\frac{\hat A_{2t}}{A_*}
\Big)
\simeq
(1.002588,\,1.193844), 
\Big(
\frac{\hat K_{2t+1}}{K_*},
\frac{\hat A_{2t+1}}{A_*}
\Big)
\simeq
(1.004624,\,0.878954),
\text{ }\forall t\geq 0.
$$

Thus, normalized capital and knowledge alternate between two states and return to the same state after two periods. 
This example shows that the cyclical mechanism can persist when $\gamma_h>0$.
\end{remark}

\subsection{Summary of findings}\label{subsec:summary-sim}
Three messages emerge. 
\begin{enumerate}
\item
\emph{Growth is semi-endogenous}: along the BGP
$G_a=G_n^{\psi_a}$ and $G_y=G_k=G_n^{\psi_k}$, so the parameters that raise the flow of research input ($\mu_a$, $s$, $\gamma$, $1-\lambda_y$) move the transition and the steady-state levels but not the long-run growth rates, as the first four rows of Table~\ref{tab:policy} confirm; R\&D-oriented reallocations buy technology growth at the cost of a temporary loss of output growth. 
\item\emph{Of the parameters we shock, only two move long-run growth}: demography, with $G_n\!\to\!1$ driving both rates to zero, and  {the knowledge elasticity, which acts through two distinct combinations: $\phi_a=(1-\gamma_a)(1-\alpha_k)-\alpha_a\gamma_h$ in the growth elasticities of Theorem \ref{theorem1}, and the total elasticity $\Gamma=\gamma_a+\alpha_a\gamma_h$ in the stability conditions of Section \ref{section-instability}, where $\phi_a$ itself reappears as (\ref{differencepositive})}. 
\item\emph{The stability margin is finite}: {the baseline $\Gamma=-0.16$ lies inside the region of Proposition~\ref{asymptotically-stable}, whereas global convergence (Theorem~\ref{global_stability}) would require $\Gamma\geq0$;} a sufficiently negative $\Gamma$ destroys stability altogether and produces the cyclical dynamics of Section~\ref{subsec:cycle-sim}, which is why policies that reduce ``idea frictions'' matter.
\end{enumerate}
\section{Conclusion and discussion}\label{conclusion}

We have developed a Solow-type semi-endogenous growth model in which technological change is generated by an R\&D sector and the existing stock of knowledge may either facilitate or hinder further innovation. We characterize the existence and uniqueness of a balanced growth path (BGP), derive its long-run growth rates, and establish conditions for global and local stability. In the semi-endogenous case, long-run technological and output growth are ultimately driven by population growth, whereas parameters affecting saving and research intensity primarily determine long-run levels.

A central contribution of the paper is to show that the innovation process itself can generate endogenous instability. When the total elasticity of new ideas with respect to knowledge,  $\Gamma\equiv\gamma_a+\alpha_a\gamma_h,$
is sufficiently negative, the BGP becomes unstable and adjustment is oscillatory. Intuitively, a high stock of knowledge can make additional ideas sufficiently difficult to produce that innovation falls sharply; the subsequent decline in knowledge then makes innovation easier again, generating overshooting in the opposite direction. Under stronger conditions, this mechanism generates deterministic cycles, including cycles of arbitrary period. These fluctuations require neither aggregate shocks nor changes in expectations. They therefore provide a distinct mechanism through which the declining research productivity emphasized by \citet{Bloom2020} can affect the dynamics of innovation and growth.

Our results also clarify the robustness of this mechanism. The Cobb--Douglas specification allows us to obtain explicit characterizations of balanced growth, comparative statics, instability, and cycles without imposing particular numerical values on the elasticities. Moreover, Theorem~\ref{global_stability_general} shows that the global-stability result extends beyond the Cobb--Douglas benchmark to a broader class of production and innovation technologies. By contrast, our explicit instability and cycle conditions exploit the power-function structure, and their extension to general technologies remains an open question.

This observation suggests several directions for future research. In particular, it would be useful to introduce more flexible technologies, such as CES specifications, to study how the substitutability or complementarity between physical capital, labor, and knowledge affects balanced growth and stability. This issue is especially relevant with the rapid development of artificial intelligence (AI), which may substitute for some tasks while complementing workers and accumulated knowledge \citep{AcemogluRestrepo2019,Acemoglu2025}. Recent evidence also points to heterogeneous productivity effects of generative AI across workers \citep{BrynjolfssonLiRaymond2025}. Allowing for non-unit elasticities of substitution would therefore help assess whether AI-related substitution dampens innovation-driven fluctuations or whether complementarities between AI, human researchers, and knowledge amplify them.

\appendix

\section{Appendix: formal proofs}\label{appendix1}

\begin{proof}[{\bf Proof of Lemma \ref{lemma1}}]
First, by the capital accumulation equation, we have $K_{t+1} = (1 - \delta_k) K_t + (1 - \gamma) s Y_t.$ Dividing both sides by $K_t$, we obtain
\begin{align}\label{Gk_Y}
    G_k = (1 - \delta_k) + (1 - \gamma) s \frac{Y_*}{K_*} \left( \frac{G_y}{G_k} \right)^t.
\end{align}

This implies immediately that $G_k - (1 - \delta_k) > 0$.  
We now claim that $G_y = G_k$. To do so, we consider two cases.
\begin{enumerate}
 \item Suppose that $G_y < G_k$. Then $\frac{G_y}{G_k} < 1$ and thus $\left(\frac{G_y}{G_k}\right)^t \to 0$ as $t \to \infty$.  
This would imply $G_k = (1 - \delta_k)$, a contradiction.

\item  Suppose that $G_y > G_k$. Then $\frac{G_y}{G_k} > 1$ and $\left(\frac{G_y}{G_k}\right)^t \to \infty$.  
This would imply $G_k - (1 - \delta_k) \to \infty$, also a contradiction.
\end{enumerate}
By combining two cases, we obtain $G_y = G_k$.

We now observe that
\begin{align*}
K_{t+1} &= (1 - \delta_k) K_t + (1 - \gamma) s A_t^{\alpha_a} K_t^{\alpha_k} L_{yt}^{\alpha_\ell}= (1 - \delta_k) K_t + (1 - \gamma) s \lambda_y^{\alpha_\ell} A_t^{\alpha_a} K_t^{\alpha_k} L_t^{\alpha_\ell}.
\end{align*}

Dividing both sides by $K_t$, we get
\begin{align}\label{Gk_intermediate}
G_k = (1 - \delta_k) + (1 - \gamma) s \lambda_y^{\alpha_\ell} A_*^{\alpha_a} L_0^{\alpha_\ell} K_*^{\alpha_k - 1}
\left( \frac{G_a^{\alpha_a} G_n^{\alpha_\ell}}{G_k^{1 - \alpha_k}} \right)^t.
\end{align}

As above, consistency of growth requires ${G_a^{\alpha_a} G_n^{\alpha_\ell}}={G_k^{1 - \alpha_k}}$.

We now prove \eqref{gan}. We have
\begin{align}
A_{t+1} &= (1 - \delta_a) A_t + \mu_a A_t^{\gamma_a} H_{t+1}^{\gamma_h} L_{a,t+1}^{\gamma_\ell}= (1 - \delta_a) A_t + \mu_a (\gamma s)^{\gamma_h} A_t^{\gamma_a} Y_t^{\gamma_h} L_{a,t+1}^{\gamma_\ell}, \notag\\
        &= (1 - \delta_a) A_t + \mu_a (\gamma s)^{\gamma_h} A_t^{\gamma_a + \alpha_a \gamma_h} 
           K_t^{\alpha_k \gamma_h} L_{yt}^{\alpha_\ell \gamma_h} L_{a,t+1}^{\gamma_\ell}. \label{29}
\end{align}
Note that  $L_{yt}^{\alpha_\ell \gamma_h} L_{a,t+1}^{\gamma_\ell} 
= (\lambda_y L_t)^{\alpha_\ell \gamma_h} (\lambda_a L_t G_n)^{\gamma_\ell} = \lambda_y^{\alpha_\ell \gamma_h} \lambda_a^{\gamma_\ell} G_n^{\gamma_\ell} L_t^{\alpha_\ell \gamma_h + \gamma_\ell}.$

Substituting into (\ref{29}), we have
\begin{align}
A_{t+1} &= (1 - \delta_a) A_t 
+ \mu_a (\gamma s)^{\gamma_h} \lambda_y^{\alpha_\ell \gamma_h} \lambda_a^{\gamma_\ell} G_n^{\gamma_\ell} 
  A_t^{\gamma_a + \alpha_a \gamma_h} K_t^{\alpha_k \gamma_h} 
  L_t^{\alpha_\ell \gamma_h + \gamma_\ell}.
\end{align}

Divide both sides by $A_t$:
\begin{equation*}
G_a = (1 - \delta_a) + \mu_a (\gamma s)^{\gamma_h} \lambda_y^{\alpha_\ell \gamma_h} \lambda_a^{\gamma_\ell} G_n^{\gamma_\ell} 
A_*^{\gamma_a + \alpha_a \gamma_h - 1} K_*^{\alpha_k \gamma_h} L_0^{\alpha_\ell \gamma_h + \gamma_\ell} 
\left( G_a^{\gamma_a + \alpha_a \gamma_h - 1} G_k^{\alpha_k \gamma_h} G_n^{\alpha_\ell \gamma_h + \gamma_\ell} \right)^t. 
\end{equation*}
By using the same argument used above, we get  that $G_a^{\gamma_a + \alpha_a \gamma_h - 1} \, G_k^{\alpha_k \gamma_h} \, G_n^{\alpha_\ell \gamma_h + \gamma_\ell} = 1.$

By combining with ${G_a^{\alpha_a} G_n^{\alpha_\ell}}={G_k^{1 - \alpha_k}}$, we obtain
\begin{equation}
G_k = G_n^{\frac{\alpha_\ell}{1 - \alpha_k}} G_a^{\frac{\alpha_a}{1 - \alpha_k}}.
\end{equation}
Finally, we have
\begin{subequations}
\begin{align}
1 &= G_a^{\gamma_a + \alpha_a \gamma_h - 1} G_k^{\alpha_k \gamma_h} G_n^{\alpha_\ell \gamma_h + \gamma_\ell}= G_a^{\gamma_a + \alpha_a \gamma_h - 1} \left( G_n^{\frac{\alpha_\ell}{1 - \alpha_k}} G_a^{\frac{\alpha_a}{1 - \alpha_k}} \right)^{\alpha_k \gamma_h} G_n^{\alpha_\ell \gamma_h + \gamma_\ell}  \\
&= G_a^{\gamma_a + \alpha_a \gamma_h - 1 + \frac{\alpha_a}{1 - \alpha_k} \alpha_k \gamma_h} G_n^{\frac{\alpha_\ell}{1 - \alpha_k} \alpha_k \gamma_h + \alpha_\ell \gamma_h + \gamma_\ell}.
\end{align}
\end{subequations}
Observe that 
\begin{subequations}
\begin{align}
\gamma_a + \alpha_a \gamma_h - 1 + \frac{\alpha_a}{1 - \alpha_k} \alpha_k \gamma_h&=\frac{\alpha_a\gamma_h-(1-\alpha_k)(1-\gamma_a)}{1-\alpha_k},\\
        \frac{\alpha_\ell}{1 - \alpha_k} \alpha_k \gamma_h + \alpha_\ell \gamma_h + \gamma_\ell&=\frac{\alpha_{\ell}\gamma_h+\gamma_{\ell}(1-\alpha_k)}{1-\alpha_k}.
\end{align}\end{subequations}
So, we obtain \eqref{gan}.

{\bf Necessary condition}. We show that a BGP must be this sequence. The growth rates of a BGP must be determined as in Lemma \ref{lemma1}.

We now find the steady state values \( K_*, A_* \). We should find two equations. According to the proof of Lemma \ref{lemma1}, we have
\begin{subequations}
\begin{align}
G_k &= (1 - \delta_k) + (1 - \gamma)s \lambda_y^{\alpha_\ell} A_*^{\alpha_a} L_0^{\alpha_\ell} K_*^{\alpha_k - 1},\\
G_a &= (1 - \delta_a) + \mu_a (\gamma s)^{\gamma_h} (\lambda_y)^{\alpha_\ell \gamma_h} (\lambda_a G_n)^{\gamma_\ell} A_*^{\gamma_a + \alpha_a \gamma_h - 1} K_*^{\alpha_k \gamma_h} L_0^{\alpha_\ell \gamma_h + \gamma_\ell}.
\end{align}
\end{subequations}

Since \( X_k \equiv (1 - \gamma)s \lambda_y^{\alpha_\ell} L_0^{\alpha_\ell} \) and \( X_a \equiv \mu_a (\gamma s)^{\gamma_h} (\lambda_y)^{\alpha_\ell \gamma_h} (\lambda_a G_n)^{\gamma_\ell} L_0^{\alpha_\ell \gamma_h + \gamma_\ell} \), we have
\begin{equation}
G_k = (1 - \delta_k) + X_k A_*^{\alpha_a} K_*^{\alpha_k - 1}
 \text{ and } G_a = (1 - \delta_a) + X_a A_*^{\gamma_a + \alpha_a \gamma_h - 1} K_*^{\alpha_k \gamma_h}.
\end{equation}

We compute \( K_* \) using \( A_* \):
\[
K_* = \left( \frac{X_k A_*^{\alpha_a}}{G_k - 1 + \delta_k} \right)^{\frac{1}{1 - \alpha_k}}.
\]

By substituting this in the second equation, we get
\begin{subequations}
\begin{align*}
    G_a - 1 + \delta_a = X_a \left( \frac{X_k}{G_k - 1 + \delta_k} \right)^{\frac{\alpha_k \gamma_h}{1 - \alpha_k}} A_*^{\gamma_a + \alpha_a \gamma_h - 1 + \frac{\alpha_a}{1 - \alpha_k} \alpha_k \gamma_h}\\
\Rightarrow A_*^{1 - \gamma_a - \alpha_a \gamma_h - \frac{\alpha_a}{1 - \alpha_k} \alpha_k \gamma_h} = \frac{1}{G_a - 1 + \delta_a} X_a \left( \frac{X_k}{G_k - 1 + \delta_k} \right)^{\frac{\alpha_k \gamma_h}{1 - \alpha_k}}.
\end{align*}
\end{subequations}

We finally compute \( Y_* \). By \eqref{Gk_Y}, we have $(G_k - 1 + \delta_k) K_* = (1 - \gamma)s Y_*$, which implies that
$$Y_* = \frac{(G_k - 1 + \delta_k) K_*}{(1 - \gamma)s}.$$

\end{proof}

\subsection{Proofs for Section \ref{section-stability}}
\label{section-stability-proofs}
\begin{proof}[{\bf Proof of Lemma \ref{lemma-2dim-system}}]

We have two accumulations:\begin{subequations}
\begin{align}
    K_{t+1} &= (1 - \delta_k)K_t + (1 - \gamma)s\lambda_y^{\alpha_{\ell}}A_t^{\alpha_a}K_t^{\alpha_k}L_t^{\alpha_\ell} \\
    A_{t+1} &= (1 - \delta_a)A_t + \mu_a(\gamma s)^{\gamma_h}(\lambda_y)^{\alpha_\ell \gamma_h}(\lambda_a G_n)^{\gamma_\ell}(A_t^{\gamma_a + \alpha_a \gamma_h}K_t^{\alpha_k \gamma_h}L_t^{\alpha_\ell \gamma_h + \gamma_\ell})
\end{align}\end{subequations}

From definitions of $\hat{K}_t$, $\hat{A}_t$, we have:
\begin{align*}
    G_k^{t+1} \hat{K}_{t+1} &= (1 - \delta_k)G_k^t \hat{K}_t + (1 - \gamma)s\lambda_y^{\alpha_{\ell}}(G_a^t \hat{A}_t)^{\alpha_a}(G_k^t \hat{K}_t)^{\alpha_k}(G_n^t L_0)^{\alpha_\ell} \notag \\
    G_a^{t+1} \hat{A}_{t+1} &= (1 - \delta_a)G_a^t \hat{A}_t + \mu_a(\gamma s)^{\gamma_h}(\lambda_y)^{\alpha_\ell \gamma_h}(\lambda_a G_n)^{\gamma_\ell}(G_a^t \hat{A}_t)^{\gamma_a + \alpha_a \gamma_h}(G_k^t \hat{K}_t)^{\alpha_k \gamma_h}(G_n^t L_0)^{\alpha_\ell \gamma_h + \gamma_\ell}
\end{align*}

By combining these equations with the balanced-growth conditions \eqref{Gank} and \eqref{gan}, we get that:\begin{subequations}
\begin{align}
    G_k \hat{K}_{t+1} &= (1 - \delta_k) \hat{K}_t + (1 - \gamma)s\lambda_y^{\alpha_{\ell}} \hat{A}_t^{\alpha_a} \hat{K}_t^{\alpha_k} L_0^{\alpha_\ell} \\
    G_a \hat{A}_{t+1} &= (1 - \delta_a) \hat{A}_t + \mu_a (\gamma s)^{\gamma_h}(\lambda_y)^{\alpha_\ell \gamma_h}(\lambda_a G_n)^{\gamma_\ell} \hat{A}_t^{\gamma_a + \alpha_a \gamma_h} \hat{K}_t^{\alpha_k \gamma_h} L_0^{\alpha_\ell \gamma_h + \gamma_\ell} 
\end{align}\end{subequations}
Recall that $
    X_k \equiv (1 - \gamma)s\lambda_y^{\alpha_{\ell}}L_0^{\alpha_\ell}$ and $ X_a \equiv \mu_a(\gamma s)^{\gamma_h}(\lambda_y)^{\alpha_\ell \gamma_h}(\lambda_a G_n)^{\gamma_\ell}L_0^{\alpha_\ell \gamma_h + \gamma_\ell}$. Then, we obtain the system (\ref{system-tilde}) 
\end{proof}

\begin{proof}[{\bf Proof of Lemma \ref{convergence_general}}]
Let $x_0\in \rr_{++}^n$ and consider $(x_t)_t$ determined by $x_{t+1}=f(x_t)$. 

Define 
\begin{align}
    \lambda_t&\equiv \max\{\frac{x_{1,t}}{x^*_1}, \ldots, \frac{x_{n,t}}{x^*_n}, 1\},&  \theta_t\equiv \min\{\frac{x_{1,t}}{x^*_1}, \ldots, \frac{x_{n,t}}{x^*_n}, 1\}.
\end{align}
We have $\theta_tx^*_i\leq x_{i,t}\leq \lambda_t x^*_i \text{ and }\lambda_t\geq 1\geq \theta_t \text{ } \forall i,\forall t.$
Then, we have \begin{align*}
    x_{i,t+1}&=f_i(x_t)\leq f_i(\lambda_t x^*) &\text{ (by order-preserving property)}\\
    &\leq \lambda_t  f_i(x^*) & \text{ (by concavity of $f_i$)}\\
&=\lambda_t  x^*_i.&
\end{align*}
Consequently, $\lambda_{t+1}\leq \lambda_t$ for any $t$.

By using the same argument, we have $\theta_{t+1}\geq \theta_t$ for any $t$. Therefore, both sequences $(\theta_t), (\lambda_t)$ converge. We will prove that they converge to $1$. Denote $\theta\equiv \lim_{t\to\infty}\theta_t$ and $\lambda\equiv \lim_{t\to\infty}\lambda_t$. It is clear that $\lambda\geq 1\geq \theta.$ Suppose that $\lambda>1$. As above, we have $ x_{i,t+1}=f_i(x_t)\leq f_i(\lambda_t x^*)$, and hence:
\begin{align*}
   \frac{x_{i,t+1}}{x^*_i}\leq \frac{f_i(\lambda_t x^*)}{x^*_i} \text{ } \forall i,t.
\end{align*}
This implies that $\lambda_{t+1}\leq \max\{\frac{f_1(\lambda_t x^*)}{x^*_1},\cdots,\frac{f_n(\lambda_t x^*)}{x^*_n},1\}$. Let $t$ tend to infinity, we get that 
\begin{align*}
    \lambda\leq \max\{\frac{f_1(\lambda x^*)}{x^*_1},\cdots,\frac{f_n(\lambda x^*)}{x^*_n},1\}.
\end{align*}
Since $f_i$ is strictly concave and $\lambda>1$, we have $f_i(\lambda x^*)<\lambda f_i(x^*)=\lambda x^*_i$, which implies that $\frac{f_i(\lambda x^*)}{x^*_i}<\lambda$ for any $i$. So, $$\lambda\leq\max\{\frac{f_1(\lambda x^*)}{x^*_1},\cdots,\frac{f_n(\lambda x^*)}{x^*_n},1\}<\lambda,$$
which is a contradiction.

Therefore, we have $\lambda=1$. By using the same argument, we have $\theta=1$.

Recall that $\theta_tx^*_i\leq x_{i,t}\leq \lambda_t x^*_i$. Let $t$ tend to infinity, we obtain that $\lim_{t\to\infty}x_{i,t}=x^*_i$ for any $i$.
\end{proof}
\begin{proof}[{\bf Proof of Proposition \ref{monotone-remark}}]
We need the following lemma.
\begin{lemma}\label{induction_t} Assume that $a_i\in (0,1), \gamma_a+\alpha_a\gamma_h\in (0,1)$.

 If
\begin{align}
\label{system-tilde-datet}
a_k\hat{K}_t+b_k\hat{A}_t^{\alpha_a}\hat{K}_t^{\alpha_k}\geq \hat{K}_t \text{ and } a_a\hat{A}_t+b_a\hat{A}_t^{\gamma_a+\alpha_a\gamma_h}\hat{K}_t^{\alpha_k\gamma_h}\geq \hat{A}_t,
\end{align}
then 
\begin{subequations}
\label{system-tilde-datet+1}
 \begin{align}\label{system-tilde-datet+1_k}
a_k\hat{K}_{t+1}+b_k\hat{A}_{t+1}^{\alpha_a}\hat{K}_{t+1}^{\alpha_k}\geq\hat{K}_{t+1}\\
\label{system-tilde-datet+1_a}a_a\hat{A}_{t+1}+b_a\hat{A}_{t+1}^{\gamma_a+\alpha_a\gamma_h}\hat{K}_{t+1}^{\alpha_k\gamma_h}\geq\hat{A}_{t+1}.
\end{align}
\end{subequations}
\end{lemma}
\begin{proof}
Condition (\ref{system-tilde-datet}) is equivalent to 
 \begin{align}\label{system-tilde-datet_1}
b_k\hat{A}_t^{\alpha_a}\hat{K}_t^{\alpha_k-1}\geq1-a_k, \quad 
b_a\hat{A}_t^{\gamma_a+\alpha_a\gamma_h-1}\hat{K}_t^{\alpha_k\gamma_h}\geq1-a_a.
\end{align}
while conditions (\ref{system-tilde-datet+1}) become  
 \begin{align}\label{system-tilde-datet+1_1}
b_k\hat{A}_{t+1}^{\alpha_a}\hat{K}_{t+1}^{\alpha_k-1}\geq1-a_k, \quad 
b_a\hat{A}_{t+1}^{\gamma_a+\alpha_a\gamma_h-1}\hat{K}_{t+1}^{\alpha_k\gamma_h}\geq1-a_a.
\end{align}
According to (\ref{system-tildeK}) and (\ref{system-tildeA}), we have 
\begin{subequations} \begin{align}
b_k\hat{A}_{t+1}^{\alpha_a}\hat{K}_{t+1}^{\alpha_k-1}
&=b_k\Big(a_a\hat{A}_t+b_a\hat{A}_t^{\gamma_a+\alpha_a\gamma_h}\hat{K}_t^{\alpha_k\gamma_h}\Big)^{\alpha_a}\Big(a_k\hat{K}_t+b_k\hat{A}_t^{\alpha_a}\hat{K}_t^{\alpha_k}\Big)^{\alpha_k-1}\\
&=\frac{b_k\hat{A}_t^{\alpha_a}}{\hat{K}_t^{1-\alpha_k}}\Big(a_a+b_a\hat{A}_t^{\gamma_a+\alpha_a\gamma_h-1}\hat{K}_t^{\alpha_k\gamma_h}\Big)^{\alpha_a}\Big(a_k+b_k\hat{A}_t^{\alpha_a}\hat{K}_t^{\alpha_k-1})^{\alpha_k-1}.
\end{align}
\end{subequations}
By the second inequality in (\ref{system-tilde-datet_1}),  we have  $
a_a+b_a\hat{A}_t^{\gamma_a+\alpha_a\gamma_h-1}\hat{K}_t^{\alpha_k\gamma_h}\geq a_a+1-a_a=1.$ So, it remains to prove that 
\begin{align}
\label{remain1}\frac{b_k\hat{A}_t^{\alpha_a}}{\hat{K}_t^{1-\alpha_k}}\Big(a_k+b_k\hat{A}_t^{\alpha_a}\hat{K}_t^{\alpha_k-1})^{\alpha_k-1}\geq1-a_k.
\end{align}

Consider the function  $f(x)\equiv x-(1-a_k)(a_k+x)^{1-\alpha_k}$. 
We have $f'(x)=1-(1-\alpha_k)(1-a_k)(a_k+x)^{-\alpha_k}\geq 0$ for all $x\geq 1-a_k$. So, $f(x)\geq f(1-a_k)=1-a_k-(1-a_k)=0$ for all $x\geq 1-a_k$.

 According to the first inequality in (\ref{system-tilde-datet_1}), we have $\frac{b_k\hat{A}_t^{\alpha_a}}{\hat{K}_t^{1-\alpha_k}}\geq 1-a_k$. Therefore, $f\Big(\frac{b_k\hat{A}_t^{\alpha_a}}{\hat{K}_t^{1-\alpha_k}}\Big)\geq 0$. This means that 
 \begin{align}
 \frac{b_k\hat{A}_t^{\alpha_a}}{\hat{K}_t^{1-\alpha_k}}\geq (1-a_k)\Big(a_k+\frac{b_k\hat{A}_t^{\alpha_a}}{\hat{K}_t^{1-\alpha_k}}\Big)^{1-\alpha_k}.
 \end{align}
which is indeed (\ref{remain1}). We have just proved (\ref{system-tilde-datet+1_k}). By using the similar argument and the assumption $1>\gamma_a+\alpha_a\gamma_h$, we obtain (\ref{system-tilde-datet+1_a}).
\end{proof}
We are now ready to prove Proposition \ref{monotone-remark}.   We prove the first point, as the proof of the second point is analogous.  Since (\ref{system-tilde-date0}), we have $\hat{K}_1\geq \hat{K}_0$ and $\hat{A}_1\geq \hat{A}_0$. 

Using again (\ref{system-tilde-date0}) and applying Lemma \ref{induction_t}, we have  \begin{align}\label{system-tilde-date_1}
b_k\hat{A}_{1}^{\alpha_a}\hat{K}_{1}^{\alpha_k-1}\geq 1-a_k, \quad 
b_a\hat{A}_{1}^{\gamma_a+\alpha_a\gamma_h-1}\hat{K}_{1}^{\alpha_k\gamma_h}\geq 1-a_a.
\end{align}
This implies that $\hat{K}_2\geq \hat{K}_1$ and $\hat{A}_2\geq \hat{A}_1$. 
So, by the induction argument, we get that  $\hat{K}_{t+1}\geq \hat{K}_t$ and $\hat{A}_{t+1}\geq \hat{A}_t$. 

\end{proof}

\begin{proof}[{\bf Proof of Theorem \ref{global_stability_general}}]
We prepare our proof by an intermediate step.
\begin{lemma}\label{lemma_transformation}Define 
$\hat{K}_t=K_t/L_t^{\psi_k}, \quad \hat{A}_t=A_t/L_t^{\psi_a}.$ We have
\begin{subequations}\begin{align}
\label{hatKgeneral}
G_k\hat{K}_{t+1}&=(1-\delta_k)\hat{K}_t+(1-\gamma)s\hat{F}(\hat{A}_t,\hat{K}_t), \\
\label{hatAgeneral}G_a\hat{A}_{t+1}&=(1-\delta_a)\hat{A}_t+\hat{F}_a(\hat{A}_t,\hat{K}_t).
\end{align}\end{subequations}
\end{lemma}
\begin{proof}    
First, by Assumption \ref{FF_a}.\ref{FF_a2}, we can verify that
\begin{subequations}
\begin{align*}
Y_t&=F(A_t,K_t,\lambda_yL_t)=F(\hat{A}_tL_t^{\psi_a},\hat{K}_tL_t^{\psi_k},\lambda_yL_t)
=L_t^{\psi_k}\hat{F}(\hat{A}_t,\hat{K}_t),\\
H_{t+1}&=\gamma s L_t^{\psi_k}\hat{F}(\hat{A}_t,\hat{K}_t),\\
F_a(A_t, H_{t+1},\lambda_aL_{t+1})&=F_a\Big(\hat{A}_tL_t^{\psi_a}, \gamma s L_t^{\psi_k}\hat{F}(\hat{A}_t,\hat{K}_t),\lambda_aG_nL_{t})=L_t^{\psi_a}\hat{F}_a(\hat{A}_t,\hat{K}_t).
\end{align*}
\end{subequations}
By substituting these equations in \eqref{Kgeneral} and \eqref{Ageneral}, we obtain the system \eqref{hatKgeneral}-\eqref{hatAgeneral}.
\end{proof}
We now prove Theorem \ref{global_stability_general}. Take the vector $(\hat{K}_*,\hat{A}_*)\in \rr_{++}^2$  as in Assumption \ref{FF_a}. We can check that $K_t=L_0^{\psi_k}\hat{K}_*G_k^t$, $A_t=L_0^{\psi_a}\hat{A}_*G_a^t$ are part of a BGP.

Finally, we prove the global stability. Look at the system \eqref{hatKgeneral}-\eqref{hatAgeneral}.   Thanks to Assumption \ref{FF_a} and Lemma \ref{lemma_transformation}, we can check that this system satisfies all assumptions in Lemma \ref{convergence_general}. So, applying Lemma \ref{convergence_general}, we obtain the convergence of $\hat{K}_t,\hat{A}_t$ and  $\frac{A_t}{(G_n^{\psi_a})^t},\frac{K_t}{(G_n^{\psi_k})^t}.$
\end{proof}

\subsection{Proofs for Section \ref{section-instability}}
\label{section-instability-proofs}
\begin{proof}[{\bf Proof of Proposition \ref{unstable}}]

Let us firstly compute the Jacobian matrix $J_*\equiv J(K_*,A_*)$ at the point $(K_*, A_*)$. First, for any $K>0,A>0$, we have
\[
\begin{aligned}
    J_{11}(K, A) &\equiv \frac{\partial f_1(K,A)}{\partial K}= a_k + \alpha_k b_k A^{\alpha_a} K^{\alpha_k - 1} \\
    J_{12}(K, A) &\equiv\frac{\partial f_1(K,A)}{\partial A}= \alpha_a b_k A^{\alpha_a - 1} K^{\alpha_k} \\
    J_{21}(K, A) &\equiv\frac{\partial f_2(K,A)}{\partial K}= b_a \alpha_k \gamma_h A^{\gamma_a + \alpha_a \gamma_h} K^{\alpha_k \gamma_h - 1} \\
    J_{22}(K, A) &\equiv\frac{\partial f_2(K,A)}{\partial A}= a_a + b_a (\gamma_a + \alpha_a \gamma_h) A^{\gamma_a + \alpha_a \gamma_h - 1} K^{\alpha_k \gamma_h}.
\end{aligned}
\]

At the steady state, we have 
\begin{align*}
    K_* = a_k K_* + b_k A_*^{\alpha_a} K_*^{\alpha_k} &\Leftrightarrow (1 - a_k) = b_k A_*^{\alpha_a} K_*^{\alpha_k - 1},\\
    A_* = a_a A_* + b_a A_*^{\gamma_a + \alpha_a \gamma_h} K_*^{\alpha_k \gamma_h} &\Leftrightarrow (1 - a_a) = b_a A_*^{\gamma_a + \alpha_a \gamma_h - 1} K_*^{\alpha_k \gamma_h}.
\end{align*}
Therefore, we find that
\[
\begin{aligned}
    J_{11}(K_*, A_*) &= a_k + \alpha_k (1 - a_k),&    J_{12}(K_*, A_*) &= \alpha_a (1 - a_k) \dfrac{K_*}{A_*} \\
    J_{21}(K_*, A_*) &= \alpha_k \gamma_h (1 - a_a) \dfrac{A_*}{K_*}, &
    J_{22}(K_*, A_*) &= a_a + (\gamma_a + \alpha_a \gamma_h)(1 - a_a).
\end{aligned}
\]

According to Theorem 4.11's point 2 in \citet{Elaydi2007}, the fixed point $(K_*,A_*)$ is unstable if the spectral radius $\rho(J_*)$ of the Jacobian matrix $J_*$ is strictly higher than $1$. By the definition of $\rho(J_*)$, we have 
\begin{align}
    \rho(J_*)=\max(|\lambda_1|, |\lambda_2|)
\end{align}
where $\lambda_1, \lambda_2$ are the eigenvalues of  $J_*$. 

 Let $T=\operatorname{tr}(J_*)$, $D=\det(J_*)$, and
$p(\lambda)=\lambda^2-T\lambda+D$. We see that
\[
\begin{aligned}
    \text{tr}(J(K_*, A_*)) &= J_{11}(K_*, A_*) + J_{22}(K_*, A_*) \\
    &= a_k + \alpha_k (1 - a_k) + a_a + (\gamma_a + \alpha_a \gamma_h)(1 - a_a) \\
    \text{det}(J(K_*, A_*)) &= \big(a_k + \alpha_k(1 - a_k)\big)\big(a_a + (\gamma_a + \alpha_a \gamma_h)(1 - a_a)\big) \\
    &\quad - \alpha_a (1 - a_k) \alpha_k \gamma_h (1 - a_a).
\end{aligned}
\]

\begin{enumerate}
\item Direct calculation gives
\begin{subequations}
\begin{align}
p(1)&=(1-a_k)(1-a_a)\phi_a,\label{eq:p-one}\\
p(-1)&=[1+a_k + \alpha_k (1 - a_k)][1+a_a+(1-a_a)\Gamma]
-\alpha_a\alpha_k\gamma_h(1-a_k)(1-a_a).\label{eq:p-minus-one}
\end{align}\end{subequations}
If $\phi_a<0$, then $p(1)<0$ and $\lim_{\lambda\to +\infty}p(\lambda)=+\infty$. So a real root exceeds one. This implies $\rho(J(K_*,A_*))>1$.

\item 
Observe that $|D|=|\lambda_1| |\lambda_2|$. So, $\rho(J(K_*,A_*))>1$ if  $|D|>1$. Of course, $|D|>1$ if $D<-1$ or $D>1$. Lastly,  we can compute that \begin{align*}
   D=&J_{11}(K_*, A_*)J_{22}(K_*, A_*)-J_{12}(K_*, A_*)J_{21}(K_*, A_*)\\
    =&(a_k + \alpha_k(1 - a_k))(a_a + (\gamma_a + \alpha_a \gamma_h)(1 - a_a)) - \alpha_a (1 - a_k) \alpha_k \gamma_h (1 - a_a).
\end{align*}
So, we have finished our proof.
\end{enumerate}
\end{proof}

\begin{proof}[{\bf Proof of Proposition \ref{asymptotically-stable}}]
Applying Theorems 4.4 and 4.11 in \citet{Elaydi2007}, the steady state $(K_*, A_*)$ is asymptotically stable if $|\textit{tr}(J(K_*, A_*))| - 1 < \textit{det}(J(K_*, A_*)) < 1.$ 
Conditions $\gamma_a + \alpha_a \gamma_h<1$ and  $a_k, a_a \in (0,1)$, $\alpha_k,\alpha_a\in [0,1]$ imply that 
\begin{subequations}
\begin{align}
    \text{det}(J(K_*, A_*)) &= (a_k + \alpha_k(1 - a_k))(a_a + (\gamma_a + \alpha_a \gamma_h)(1 - a_a)) \\
    &\quad - \alpha_a (1 - a_k) \alpha_k \gamma_h (1 - a_a)\\
    &\leq (a_k + \alpha_k(1 - a_k))(a_a + (\gamma_a + \alpha_a \gamma_h)(1 - a_a))<1
\end{align}
\end{subequations}
Under condition (\ref{tracepositive}), we have $\text{tr}(J(K_*, A_*)) > 0$. Moreover, we can check that
\begin{subequations}\begin{align}
   & |\text{tr}(J(K_*, A_*))| - 1 - \text{det}(J(K_*, A_*)) \\
    =& - (1 - a_k)(1 - a_a) ((1 - \alpha_k)(1 - \gamma_a) - \alpha_a \gamma_h) < 0
\end{align}\end{subequations}
because of condition (\ref{differencepositive}). We have finished our proof.

\end{proof}

\begin{proof}[{\bf Proof of Proposition  \ref{cycle}}]
In this proof, to simplify the notations, we write $a,b,\gamma$ instead of $a_a,b_a,\gamma_a$. Recall that $a_a,b_a$ are defined by (\ref{abi_definition}). 
Denote \textcolor{blue}{$g(x)=a x + b x^{\gamma}$}.

{\bf Point \ref{point_2period}.} 
It suffices to prove that there exist $(\hat{A}_0, \hat{A}_1)\in \rr_{++}^2$ such that $      \hat{A}_{1}= g(\hat{A}_0)$ and $ \hat{A}_{0}=g(\hat{A}_1)$ with $ \hat{A}_{1}\not=\hat{A}_{0}.$ 
This happens if there exists $x>0$ such that $x =g^2(x)=a (a x + b x^{\gamma})+ b (a x + bx^{\gamma})^{\gamma}$ and $x\not=g(x
)$. This leads to the following equation
\begin{subequations}
\begin{align}
     (1-a^2)x^{1-\gamma}-ab-b(a  + \frac{b}{x^{1-\gamma}} )^{\gamma}=0\\
    (1-a^2)-a\frac{b}{x^{1-\gamma}}-\frac{b}{x^{1-\gamma}}(a  + \frac{b}{x^{1-\gamma}} )^{\gamma}=0.
\end{align}
\end{subequations}
Denote \textcolor{blue}{$X\equiv a  + \frac{b}{x^{1-\gamma}}$}. When $\gamma<1$, $X$ is strictly decreasing in $x$ and $X$ has the range $(a,\infty)$. 

Then, we have an equation
\begin{align}
 G(X)\equiv    1-a^2-a(X-a)-(X-a)X^{\gamma}=0
\end{align}
Since $X\equiv a  + \frac{b}{x^{1-\gamma}}>a$, we focus on the interval $X\in (a,\infty)$. Notice that $a\in (0,1)$.

Observe that $G(1)=0$. Moreover, $X=1$ if and only if $a  + \frac{b}{x^{1-\gamma}}=1$ (or, equivalently, $x=g(x)$, which is equivalent to $x=(\frac{b}{1-a})^{\frac{1}{1-\gamma}}$). 

We have $G(0)=+\infty, G(a)=1-a^2>0, G(\infty)=-\infty$ because $G(X)=1-a^2-(X-a)(a+X^{\gamma})$.

Let us compute 
\begin{align}
G^{\prime}(X)&=-a-X^{\gamma}-(X-a)\gamma X^{\gamma-1}=-a-(1+\gamma)X^{\gamma}+a\gamma X^{\gamma-1}.
\end{align}
We consider two cases.
\begin{enumerate}
\item If $-1\leq \gamma <0$, then we have $G^{\prime \prime}(X)>0$ and $G^{\prime}(X)<0$ for any $X>0$. So, the equation $G(X)=0$ has a unique solution which is 1. In this case, there is no two-cycle.
\item If $\gamma <-1$. 
 We compute $ G'(1)=-a-1-(1-a)\gamma$.

Assume that $G'(1)>0$, (i.e., $\gamma<-\frac{1+a}{1-a}$). Then there exists $\epsilon\in(0,1)$ such that  $G'(X)>0$ for any $X\in (1-\epsilon,1+\epsilon)$. Consequently, for $X'\in (1,1+\epsilon)$, by the mean value theorem, we have $G(X')-G(1)=G'(\xi)(X'-1)>0$ where $\xi \in (1,X')$. So, $G(X')>0$. Since $G(\infty)=-\infty$, there exists $X_0\in (X',\infty)$ such that $G(X_0)=0$. 

We have $X_0>1>a$ because $X_0>X'>1$. We see that there exists $X_0\not= 1$ such that $X_0>a$ and  $G(X_0)=0$.

Now, we define $\hat{A}_0$ by $X_0=a  + \frac{b}{(\hat{A}_0)^{1-\gamma}}$ and $\hat{A}_1=g(\hat{A}_0)$. Of course, $\hat{A}_1\not=\hat{A}_0$. Indeed, if $\hat{A}_1=\hat{A}_0$, then $a+b\hat{A}_0^{\gamma-1}=1$. That is $X_0=1$, a contradiction.
\end{enumerate}

Last, we investigate condition $\gamma<-\frac{1+a}{1-a}$ in the second case ($\gamma<-1$). Note that the notation $a$ is actually $a_a$ and $\gamma$ is $\gamma_a$. In the case $\gamma_h=0$, we have
\begin{align}  \label{parameter_gammah_0}
G_a&= G_n^\frac{ \frac{\alpha_\ell}{1 - \alpha_k} \alpha_k \gamma_h + \alpha_\ell \gamma_h + \gamma_\ell } { 1 - \gamma_a - \alpha_a \gamma_h - \frac{\alpha_a}{1 - \alpha_k} \alpha_k \gamma_h }=G_n^{\frac{\gamma_\ell}{1-\gamma_a}}, \quad 
 a_a\equiv \frac{1-\delta_a}{G_a}=  \frac{1-\delta_a}{G_n^{\frac{\gamma_\ell}{1-\gamma_a}}}.
\end{align}


Consequently, 
\begin{align}\label{conditiongamma_a}
    \gamma_a<-\frac{1+a_a}{1-a_a}\Leftrightarrow \gamma_a+1+a_a(1-\gamma_a)<0 \Leftrightarrow \gamma_a+1+(1-\gamma_a)\frac{1-\delta_a}{G_n^{\frac{\gamma_\ell}{1-\gamma_a}}}<0
\end{align}

We now consider two cases.
\begin{enumerate}
    \item When $G_n\geq 1,$ then $G_n^{\frac{\gamma_\ell}{1-\gamma_a}}\geq 1$. Consequently,  condition (\ref{conditiongamma_a}) holds if $\gamma_a+1+(1-\gamma_a)(1-\delta_a)<0$, or, equivalently, $\gamma_a<-\frac{2-\delta_a}{\delta_a}$.

\item When $G_n<1$, then $(\frac{1}{G_n})^{\frac{\gamma_\ell}{1-\gamma_a}}$ is strictly increasing in $\gamma_a$. Moreover, $\frac{\gamma+1}{1-\gamma}=-1+\frac{2}{1-\gamma}$  is strictly increasing in $\gamma_a$. Consequently, condition (\ref{conditiongamma_a}) is equivalent to $\gamma_a<x\equiv \gamma(G_n,\gamma_\ell,\delta_a)$ where $x$ is the unique solution to the equation 
$$x+1+(1-x)\frac{1-\delta_a}{G_n^{\frac{\gamma_\ell}{1-x}}}=0.$$
\end{enumerate}
By combining two cases, we see that condition (\ref{conditiongamma_a}) holds if $\gamma_a$ is low enough.
 
{\bf Point \ref{point_Nperiod}.}  By the Li-Yorke Theorem \citep{LiYorke1975},\footnote{See also \citet{Elaydi2007}, Theorem 2.9, page 96.} it suffices to prove that there exists $x_0$ satisfying $x_0=g^3(x_0)$, $x_0\not= g(x_0)$ and $x_0\not =g^2(x_0)$.

 We have 
\begin{align*}
    g^2(x)&=g(g(x))=a(a x + b x^{\gamma})+b(a x + b x^{\gamma})^{\gamma}\\
    x-g^3(x)&=x-a\big[a(a x + b x^{\gamma})+b(a x + b x^{\gamma})^{\gamma}]-b\big[a(a x + b x^{\gamma})+b(a x + b x^{\gamma})^{\gamma}]^{\gamma}\\
    &=(1-a^3)x-a^2bx^{\gamma}-ab(a x + b x^{\gamma})^{\gamma}-b\big[a(a x + b x^{\gamma})+b(a x + b x^{\gamma})^{\gamma}]^{\gamma}
\end{align*}
Since $x>0$, condition $x-g^3(x)=0$ is equivalent to
\begin{align*}
 (1-a^3)x^{1-\gamma}-a^2b-ab\big(a+\frac{b}{x^{1-\gamma}}\big)^{\gamma}-b\Big[a\big(a+\frac{b}{x^{1-\gamma}}\big)+\frac{b}{x^{1-\gamma}}\big(a+\frac{b}{x^{1-\gamma}}\big)^{\gamma}\Big]^{\gamma}=0\\
 (1-a^3)-a^2\frac{b}{x^{1-\gamma}}-a\frac{b}{x^{1-\gamma}}\big(a+\frac{b}{x^{1-\gamma}}\big)^{\gamma}-\frac{b}{x^{1-\gamma}}\big(a+\frac{b}{x^{1-\gamma}}\big)^{\gamma}\Big[a+\frac{b}{x^{1-\gamma}}\frac{1}{\big(a+\frac{b}{x^{1-\gamma}}\big)^{1-\gamma}}\Big]^{\gamma}=0.
\end{align*}

Again, denote $X\equiv a  + \frac{b}{x^{1-\gamma}}$. We have an equation
\begin{align}
G_2(X)&\equiv 1-a^2X-a(X-a)X^{\gamma}-(X-a)\Big(aX+(X-a)X^{\gamma}\Big)^{\gamma}=0.
\end{align}
It is clear that $G_2(1)=0$ or, equivalently, $x^*=g^3(x^*)$, where $x^*$ is the unique fixed point determined by $x^*=g(x^*)$, i.e.,  $(1-a)(x^*)^{1-\gamma}=b$.

Moreover, $\lim_{X\to\infty}G_2(X)=-\infty.$

We will provide a sufficient condition for the existence of $x_0$ satisfying $x_0=g^3(x_0)$, $x_0\not= g(x_0)$ and $x_0\not =g^2(x_0)$. To prove this claim, we need some intermediate steps. First, we have the following result.
\begin{lemma}\label{step1}
If there exists $X>1$ such that $G_2(X)>0$, then there exists $x_0$ satisfying $x_0=g^3(x_0)$, $x_0\not= g(x_0)$ and $x_0\not =g^2(x_0)$.
\end{lemma}
\begin{proof}
Since $\lim_{x\to\infty}G_2(x)=-\infty$ and $G_2(X)>0$, there exists $X_0\in (1,\infty)$ such that $G(X_0)=0$.  Since $X_0>1$, we can take the unique $x_0$ determined by $X_0=a+\frac{b}{x_0^{1-\gamma}}$.   Then, define $\hat{A}_0=x_0,\hat{A}_t=x_t\equiv g^t(x_0)$. It remains to prove that $x_0\not=g(x_0)$ and $x_0\not=g^2(x_0)$. Moreover, if $x_0=g^2(x_0)$, then $g(x_0)=g^3(x_0)=x_0$. So, it suffices to prove that $x_0\not=g(x_0)$. Suppose that $x_0=g(x_0)$, then $x_0=x^*$ where $(1-a)(x^*)^{1-\gamma}=b$. This implies that $X_0=a+\frac{b}{x_0^{1-\gamma}}=1$, a contradiction. 
\end{proof}

\begin{lemma} \label{step2}Let $X$ be such that $\frac{1}{a^2}>X>\frac{1}{a}$ (such an $X$ exists because $a\in (0,1)$). Let $\gamma$ be negative and low enough so that 
\begin{align}\label{boundgamma1} \frac{1-a^2X}{(a+1)(X-a)}>\big(aX\big)^{\gamma}
\end{align}
Then, $X>1$ and $G(X)>0$.
\end{lemma}
\begin{proof}

Consider $G_2(X)=1-a^2X-a(X-a)X^{\gamma}-(X-a)\Big(aX+(X-a)X^{\gamma}\Big)^{\gamma}$.  Observe that
\begin{align*}
&a(X-a)X^{\gamma}+(X-a)\Big(aX+(X-a)X^{\gamma}\Big)^{\gamma}\leq a(X-a)X^{\gamma}+(X-a)(aX)^{\gamma}\\
&\leq a(X-a)(aX)^{\gamma}+(X-a)(aX)^{\gamma} \quad \quad \quad \text{ (because $1<a^{\gamma}$)}\\
&=(a+1)(X-a)(aX)^{\gamma}=(a+1)(X-a)(aX)^{\gamma}\\
&\leq 1-a^2X  \text{ (thanks to  (\ref{boundgamma1})).}
\end{align*}
Therefore, $G_2(X)>0$.
\end{proof}

We now look at condition (\ref{boundgamma1}). Recall that $a=a_a= \frac{1-\delta_a}{G_n^{\frac{\gamma_\ell}{1-\gamma_a}}}$ and $\gamma=\gamma_a$ (see (\ref{parameter_gammah_0})). 

\begin{lemma}\label{lemma_gamma_a}
Let $X=\frac{1}{2}(\frac{1}{a}+\frac{1}{a^2})$. Then, $X>1$ and $\frac{1}{a^2}>X>\frac{1}{a}$. Assume that $a=a_a= \frac{1-\delta_a}{G_n^{\frac{\gamma_\ell}{1-\gamma_a}}}$ satisfies
\begin{align}\frac{a^2}{(a+1)(1+2a+2a^2)}>\big(\frac{a+1}{2a}\big)^{\gamma_a}.\label{key_gamma}
\end{align}
Then $X>1$ and $G(X)>0$.
\end{lemma}

\begin{proof}(\ref{boundgamma1}) is equivalent to
\begin{subequations}
\begin{align}
   \frac{1-a^2X}{(a+1)(X-a)}>\big(aX\big)^{\gamma_a}
\Leftrightarrow   \frac{1-\frac{a+1}{2}}{(a+1)    \big(\frac{1+a}{2a^2}-a\big)}>\big(\frac{a+1}{2a}\big)^{\gamma_a}\\
   \Leftrightarrow \frac{a^2}{(a+1)(1+2a+2a^2)}>\big(\frac{a+1}{2a}\big)^{\gamma_a}.
\end{align}\end{subequations}
\end{proof}
According to Lemmas \ref{step1}, \ref{step2}, \ref{lemma_gamma_a}, when (\ref{key_gamma}) holds, there exists $x_0$ satisfying $x_0=g^3(x_0)$, $x_0\not= g(x_0)$ and $x_0\not =g^2(x_0)$. In other words, there exists a three-period cycle and, hence,  a $N-$period cycle.


\end{proof}

{\small
\bibliographystyle{apalike}
\bibliography{References.bib}
}
\end{document}